\documentclass[opre,nonblindrev]{informs4_notitle}

\OneAndAHalfSpacedXI 

\usepackage{endnotes}
\let\footnote=\endnote
\let\enotesize=\normalsize
\def\notesname{Endnotes}%
\def\makeenmark{$^{\theenmark}$}
\def\enoteformat{\rightskip0pt\leftskip0pt\parindent=1.75em
	\leavevmode\llap{\theenmark.\enskip}}

\usepackage{natbib}
	\bibpunct[, ]{(}{)}{,}{a}{}{,}%
	\def\bibfont{\fontsize{8}{9.5}\selectfont}%
\TheoremsNumberedThrough     
\ECRepeatTheorems

\EquationsNumberedThrough    

\gdef\AQ#1{}
\gdef\CQ#1{}

\usepackage{amsmath}
\usepackage{amsbsy}
\usepackage{bm,bbm}
\usepackage{hyperref} 
\hypersetup{
	colorlinks,
	linkcolor={red!50!black},
	citecolor={blue!50!black},
	urlcolor={blue!80!black}
}
\usepackage{mathrsfs}
\usepackage{nicefrac,xfrac}
\usepackage{xcolor}
\usepackage{subcaption}
\usepackage{upgreek}
\usepackage{enumitem}
\usepackage{comment}
\usepackage{lscape}
\usepackage{pbox}
\usepackage{tikz}
    \usetikzlibrary{arrows}    

\usepackage{algorithm}
\usepackage{algpseudocode}
\makeatletter
\algrenewcommand\alglinenumber[1]{\tiny #1:}
\algrenewcommand\ALG@beginalgorithmic{\small} 
\algnewcommand{\LineComment}[1]{\State \(\triangleright\) #1} 
\algrenewcommand{\algorithmiccomment}[1]{\hskip3em$\triangleright$ #1} 
\renewcommand{\Comment}[1]{\hskip3em$\triangleright$ #1} 
\makeatother
\algtext*{EndIf}   
\algtext*{EndFor}  
\algtext*{EndWhile}  
\algtext*{EndProcedure}  

\newcommand{\nld}{n_l}
\newcommand{\x}{x}
\newcommand{\xvec}{\mathbf{\x}}
\newcommand{\xvecset}{\mathcal{X}}
\newcommand{\indexset}{\mathcal{J}}

\newcommand{\nfl}{n_f}
\newcommand{\y}{y}
\newcommand{\yvec}{\mathbf{\y}}
\newcommand{\Al}{\mathbf{A}}
\newcommand{\Bf}{\mathbf{B}}
\newcommand{\rhsf}{\mathbf{b}}

\newcommand{\Gx}{\mathbf{G}_{\xvec}}

\newcommand{\hvec}{\mathbf{h}}

\newcommand{\state}{\mathbf{s}}
\newcommand{\stateset}{\mathcal{S}}
\newcommand{\fstateset}{\hat{\mathcal{S}}}
\newcommand{\issvec}{\upgamma}
\newcommand{\iss}{\gamma}

\newcommand{\cvec}{\mathbf{c}}
\newcommand{\dvec}{\mathbf{d}}
\newcommand{\al}{\mathbf{a}}

\newcommand{\z}{z}
\newcommand{\extval}{\mathscr{G}}
\newcommand{\polynet}{\mathscr{P}}
\newcommand{\net}{\mathcal{N}}
\newcommand{\nodeset}{\mathscr{U}}
\newcommand{\edgeset}{\mathscr{E}}
\newcommand{\edgevalvec}{\upnu}
\newcommand{\edgeval}{\nu}
\newcommand{\node}{u}

\newcommand{\edge}{e}
\newcommand{\initnode}{r}
\newcommand{\termnode}{t}

\newcommand{\flow}{\omega}
\newcommand{\flowvec}{\upomega}
\newcommand{\outedges}{\edgeset^{+}}
\newcommand{\inedges}{\edgeset^{-}}

\newcommand{\Gy}{\mathbf{G}_{\yvec}}

\newcommand{\symgroup}{\hat{\stateset}}

\newcommand{\budget}{B}
\newcommand{\Node}{U}

\newcommand{\nodelb}{\ell}
\newcommand{\nodelbvec}{\boldsymbol{\ell}}
\newcommand{\nodeub}{\mathfrak{r}}
\newcommand{\nodeubvec}{\boldsymbol{\mathfrak{r}}}

\newcommand{\ncons}{m}
\newcommand{\nf}{n_f}

\newcommand{\yvecset}{\mathcal{Y}}

\newcommand{\val}{\phi}

\newcommand{\zeros}{\mathbf{0}}
\newcommand{\vale}{\overline{\val}}

\newcommand{\statevec}{\mathbf{s}}

\newcommand{\augindexset}{\hat{\indexset}}

\newcommand{\MibS}{\texttt{MibS}}
\newcommand{\BC}{\texttt{B\&C}}
\newcommand{\DD}{\texttt{DD}}
\newcommand{\DDMaxMin}{\texttt{DD+MaxMin}}
\newcommand{\DDMaxMinSep}{\texttt{DD+MaxMin+Sep}}
\newcommand{\DDMibS}{\texttt{MibS+DD}}
\newcommand{\DDBC}{\texttt{B\&C+DD}}

\begin{document}
	
\def\COPYRIGHTHOLDER{INFORMS}%
\def\COPYRIGHTYEAR{2024}%
\def\DOI{\fontsize{7.5}{9.5}\selectfont\sf\bfseries\noindent no-doi}

\RUNAUTHOR{Lozano et~al.} %

\RUNTITLE{Network Relaxations for Discrete Bilevel Optimization under Linear Interactions}

\TITLE{Network Relaxations for Discrete Bilevel Optimization under Linear Interactions}


\ARTICLEAUTHORS{

	\AUTHOR{Leonardo Lozano}
	\AFF{Operations, Business Analytics \& Information Systems, University of Cincinnati \\ 2925 Campus Green Drive, Cincinnati, OH 45221\\ 
		\EMAIL{leolozano@uc.edu} \URL{}}
		
	\AUTHOR{David Bergman}
	\AFF{Department of Operations and Information Management, University of Connecticut \\2100 Hillside Rd, Storrs, CT 06268\\ \EMAIL{david.bergman@uconn.edu}
	}

	\AUTHOR{Andre A. Cire}
	\AFF{Dept. of Management, University of Toronto Scarborough and Rotman School of Management \\ Toronto, Ontario M1C 1A4, Canada
	\\ \EMAIL{andre.cire@rotman.utoronto.ca} 
		}


}
	

	\ABSTRACT{
We investigate relaxations for a class of discrete bilevel programs where the interaction constraints linking the leader and the follower are linear. Our approach reformulates the upper-level optimality constraints by projecting the leader's decisions onto vectors that map to distinct follower solution values, each referred to as a state. Based on such a state representation, we develop a network-flow linear program via a decision diagram that captures the convex hull of the follower's value function graph, leading to a new single-level reformulation of the bilevel problem. We also present a reduction procedure that exploits symmetry to identify the reformulation of minimal size. For large networks, we introduce parameterized relaxations that aggregate states by considering tractable hyperrectangles based on lower and upper bounds associated with the interaction constraints, and can be integrated into existing mixed-integer bilevel linear programming (MIBLP) solvers. 
Numerical experiments suggest that the new relaxations, whether used within a simple cutting-plane procedure or integrated into state-of-the-art MIBLP solvers, significantly reduce runtimes or solve additional benchmark instances. Our findings also highlight the correlation between the quality of relaxations and the properties of the interaction matrix, underscoring the potential of our approach in enhancing solution methods for structured bilevel optimization instances.
}



\AREAOFREVIEW{Optimization.}

	
	%
	
\maketitle


\section{Introduction}
\label{sec:introduction}


Often prescriptive optimization problems take the form of a two-player hierarchical process, i.e., where decisions made at an upper level by a so-called leader directly impact the possible choices that can be made at a lower level by a follower. Such lower level decisions, in turn, also influence outcomes and pay-offs at the higher level, resulting in challenging bilevel formulations that capture the complex sequential interplay between the leader and the follower. Bilevel optimization has been a pervasive research topic due to its broad array of applications, with examples in logistics design \citep{qi2022sequential}, pricing \citep{labbe1998bilevel,chen2018complexity}, energy markets \citep{arroyo2010bilevel}, and transportation \citep{kinay2023charging}, to name a few.

In this paper, we investigate relaxations for a general class of discrete bilevel programs where the interaction constraints relating the follower and leader are linear, that is,
\begin{subequations}
	\label{model:bilevel}
	\begin{align}
    \min_{\xvec, \yvec}  
        &\quad 
        f(\xvec, \yvec) \label{bilevel:obj} \\
    \textnormal{s.t.}
        &\quad
            \yvec \in \argmin_{\yvec'} 
            \left\{ 
                g(\yvec') 
                \,\colon\,
                \Al \xvec + \Bf \yvec' \ge \rhsf, \;\; \yvec' \in \yvecset
            \right\}, \label{bilevel:cons:1} \\ 
        &\quad
            \xvec \in \xvecset(\yvec) \cap \mathbb{Z}^{\nld} \label{bilevel:cons:2}.
\end{align}
\end{subequations}
More precisely, $\yvec$ is the follower's decision vector for some feasible set $\yvecset \subseteq \mathbb{R}^{\nf}$, $\nf > 0$;
$\xvec$ is the leader's decision vector for some feasible set $\xvecset(\yvec) \subseteq \mathbb{R}^{\nld}$, $\nld > 0$, which is parameterized by $\yvec$; $\Al \in \mathbb{R}^{\ncons \times \nld}, \Bf \in \mathbb{R}^{\ncons \times \nf}$ are rational coefficient matrices associated with $\ncons > 0$ linear interaction constraints; and $f, g$ are the leader's and follower's objective function, respectively. We refer to \eqref{bilevel:cons:1} and \eqref{bilevel:cons:2} as the follower's subproblem and the leader's constraints, respectively. In particular, the set $\xvecset(\yvec)$ may include additional continuous variables, and the follower's variables $\yvec$ are not required to be integers; i.e., the model imposes that the leader's variables $\xvec$ associated with the interaction constraints are integers.

Structured cases of \eqref{model:bilevel} include classical bilevel problems such as network interdiction \citep{MortonEtal07, cappanera2011, hemmati2014, lozano2017b}, blocking problems \citep{bazgan2011, Mahdavi2014}, and other models where the leader's variables are fully or partially combinatorial \citep{costa2011, Caprara2016, zare2018}. Recent developments within this class have particularly focused on general solution approaches for mixed-integer bilevel linear programs (MIBLP), where $f,g$ are linear, $\yvecset$ is polyhedral, and  
\begin{align}
	\label{cons:linearleader}
	\xvecset(\yvec) = \{ \xvec \in \mathbb{R}^{\nld} \,\colon\, \Gx \xvec + \Gy \yvec \ge \hvec \}
\end{align} 
for given matrices $\Gx,\Gy$ and vector $\hvec$ of appropriate dimensions \citep{fischetti2016intersection,fischetti2017new,tahernejad2020branch,kleinert2021survey}. MIBLP problems remain notoriously difficult to address computationally; in the now standard bilevel benchmark introduced by \cite{fischetti2016intersection}, instances with as many as 29 leader variables are not solved to optimality by state-of-the-art solvers under practical time limits (see \S\ref{sec:numericalstudy}). 

In this work, we propose an alternative but complementary perspective to relax \eqref{model:bilevel} that leverages the inherent symmetry that results from the integrality of $\xvec$. The relaxation can be used to provide optimization bounds or be incorporated directly as an initial model in existing bilevel solvers to strengthen their formulations, particularly those based on mathematical programming. Specifically, our starting point is the equivalent value-function model \eqref{model:vfr} introduced by \cite{fortuny1981representation},
\begin{subequations}
	\label{model:vfr}
	\begin{align}
		\min_{\xvec, \yvec}  
			&\quad 
			f(\xvec, \yvec) \label{vfr:obj} \\
		\textnormal{s.t.}
			&\quad
				\xvec \in \xvecset(\yvec) \cap \mathbb{Z}^{\nld}, \,\, \yvec \in \yvecset, 
					\label{vfr:cons:1}\\
			&\quad
				\Al \xvec + \Bf \yvec \ge \rhsf
					\label{vfr:cons:2}\\
			&\quad
				g(\yvec) \le \val(\xvec), \label{vfr:cons:3}
	\end{align}
\end{subequations}
where 
$
	\val(\xvec) := \min_{\yvec} 
	\left\{ 
		g(\yvec) 
		\,\colon\,
		\Bf \yvec \ge \rhsf - \Al \xvec, \;\; \yvec \in \yvecset
	\right\}
$
is a value function that provides the optimal solution value of the follower's subproblem for a given $\xvec$. In such a model, the leader is optimistic in that it assumes the follower picks the decision $\yvec$ that minimizes $f(\xvec, \yvec)$ for any $\xvec$ if multiple optima to \eqref{bilevel:cons:1} exist. State-of-the-art methods for bilevel optimization typically relax inequality \eqref{vfr:cons:3} to derive the so-called high-point relaxation, which is subsequently strengthened via separation methods that approximate $\val$ sequentially \citep{moore1990mixed, outrata1990numerical, fischetti2017new}, which is now pervasive in current approaches.

We investigate a lifted version of the high-point relaxation that explicitly encodes evaluations of $\val$ into its formulation. Our methodology is based on the observation that the value function $\val$ may be redundant in a large portion of the leader's feasible space $\xvecset(\yvec)$; more precisely,
\begin{align}
	\val(\xvec) = \val(\xvec'), \;\; \forall \xvec,\xvec' \,\,\textnormal{such that}\,\, \Al \xvec = \Al \xvec'.
\end{align}
In other words, it suffices to evaluate $\phi$ only with respect to the linear projection $\{ \statevec \in \mathbb{R}^{\ncons} \colon \statevec = \Al \xvec, \, \forall \xvec \}$ when solving \eqref{model:vfr}. Building upon this principle, our main contributions are as follows.
\begin{enumerate}
	\item Using the discrete structure of $\xvec$, we propose an extended reformulation of $\val$ that constructs the vectors $\Al \xvec$ by considering each component $\x_j$ of $\xvec$ at a time. The partial evaluations of this sequential steps are referred to as ``states'' that define symmetry classes in terms of the evaluations of $\val$.

	\item Through this state-based model, we propose a network encoding of $\val$ that identifies further redundant states via its topological structure. Specifically, the network is a decision diagram that captures the distinct evaluations of $\val$ in terms of $\xvec$ and its terminal nodes, and it is free of symmetry and unique.
	
    \item We obtain a reformulation of \eqref{model:vfr} where the value-function constraint \eqref{vfr:cons:3} is written as a flow model over the network encoding of $\val$, which is convex-hull defining with respect to the feasible solution and value pairs $(\xvec, \val(\xvec))$. We also discuss the size of such reformulations; in particular, if the followers' subproblem is defined by a knapsack constraint, one can obtain an exact single-level model of \eqref{model:vfr} with pseudo-polynomial size in the right-hand side of the constraint.
    
    \item If the network is large, we propose a new class of parameterized high-point relaxations based on aggregating states to generate smaller networks. Specifically, the aggregated recursive model provides a bound of $\val$ at all points $\xvec$ by operating on hyperrectangles of reachable states. We also demonstrate how to strengthen the approximation by solving a sequence of tractable robust optimization models with respect to the follower's subproblem \eqref{bilevel:cons:2}. 
\end{enumerate}

\smallskip
We report experiments comparing our relaxations and solution approaches with respect to two state-of-the-art MIBLP solvers -- MibS by \cite{tahernejad2020branch} and the branch-and-cut model by \cite{fischetti2017new} -- on structured instances and standard bilevel benchmark problems with up to 3,000 variables. Our numerical results suggest that the quality of the proposed high-point relaxations is correlated with properties of the interaction matrix $\Al$, such as sparsity and the range of coefficients. For these cases, the relaxations matched the optimal solution bound for several structured and benchmark instances. We also report runtimes of an exact, straightforward cutting-plane procedure that iteratively improves the proposed relaxation, as well as the performance gains obtained when incorporating the proposed flow models as redundant constraints into the two MIBLP solvers. Specifically, for benchmark instances, the combined relaxation improved the average runtimes of the branch-and-cut model by \cite{fischetti2017new}, the current best-performing solver for this class, by 11.7\%. Further, the cutting-plane approach solved to optimality several challenging instances that presented positive gaps for other solvers, reducing the best average runtime for the benchmark by 16\%.

\medskip
\paragraph{Assumptions.} 

Model \eqref{model:bilevel} imposes that interaction constraints are limited to leader's integer variables, which is a standard condition that ensures optimal solutions are attainable by enumerative branching methods \citep{koppe2010parametric}. Throughout this text, we also consider the following assumptions for the existence of solutions, tractability of the bilevel program, and ease of exposition. 

\begin{assumption} 
    \label{as:subproblemcalls}
    The evaluation of $\val(\xvec)$ for a given $\xvec$ is computationally inexpensive.
\end{assumption}
Assumption \ref{as:subproblemcalls} indicates that, even if the follower's subproblem belongs to a difficult complexity class (e.g., NP-Hard), we are equipped with solvers that can evaluate or approximate $\val$ in a short computational time for instances of interest.  This is a fundamental primitive of optimization methods based on separating infeasible solutions of the high-point relaxation \citep{fischetti2017new}, which we also rely in the development of our methodology.  Thus, our primarily focus is on reducing the calls to $\val$ when evaluating \eqref{model:vfr}.

\begin{assumption} 
    \label{as:binarylinearinteger} 
	In all our formulations,
    \begin{itemize}    
		\item[(a)] The bilevel problem \eqref{model:bilevel} is an MIBLP where $\xvec$ are binaries.  
		\item[(b)] If the follower's subproblem \eqref{bilevel:cons:1} is feasible for $\xvec' \in \{0,1\}^{\nld}$, then it is also bounded for $\xvec'$, i.e., an optimal $\yvec^*$ that solves \eqref{bilevel:cons:1} exists.	
    \end{itemize}
\end{assumption}
Bilevel problems satisfying Assumption \ref{as:binarylinearinteger}-(a) are such that their high-point relaxations can be solved using standard mixed-integer linear programming models. Assumption \ref{as:binarylinearinteger}-(b) imposes that the follower's subproblems are bounded for any $\xvec$. These conditions capture the large majority of benchmarks in bilevel optimization, and we focus on such problem classes for our methodological development and computational results. That is, we are particularly interested in formulations that could also lead to strong linear programming bounds. The assumption of binary leader's variables, in particular, facilitates exposition since problems with integer variables are easily transformed into 0-1 problems using a standard binary expansion.

\medskip 
\paragraph{Organization.} The paper is organized as follows. Section \ref{sec:relatedwork} extends our literature review in bilevel optimization and associated relaxations in mathematical programming. Section \ref{sec:extension} describes an extended reformulation of \eqref{model:vfr} based on the graph of the function $\val$ and presents an state-based models. Section \ref{sec:value} introduces the network encoding and the underlying convex-hull reformulation as a flow model. Section \ref{sec:networkconstruction} develops a general methodology to build the network encoding and remove symmetric nodes. Section \ref{sec:approximation} develops parameterized approximation strategies to generate smaller reformulations. Section \ref{sec:numericalstudy} presents a numerical study comparing and integrating the approach into existing MIBLP solvers. Finally, we conclude in Section \ref{sec:conclusion} with extensions, trade-offs, and future work. Proofs are included in the Appendix \cite{app:proofs}.

\medskip 
\paragraph{Notation.} In this paper, vectors and matrices are represented in bold, and scalars in regular font. For ease of notation, we fix $\val(\xvec) = +\infty$ if the follower's subproblem is infeasible for $\xvec$. The $j$-th column of $\Al$ is denoted by $\al_{j}$, and $\zeros$ is the vector or matrix of zeros in a suitable dimension. We also refer to $\val(\xvec)$ as the follower's value of $\xvec$. A vector $\mathbf{v}$ with indices defined by a set $S$ is written as $\mathbf{v} = \{v_i\}_{i \in S}$.


\section{Related Work}
\label{sec:relatedwork}

Bilevel optimization was first introduced in the seminal work by \cite{bracken1973mathematical} and has now become a pervasive research area in operations and economics. We refer to \cite{dempe2015bilevel} and \cite{colson2007overview} for methodological foundations, recent surveys, and applications.

This work investigates general-purpose exact and relaxation mechanisms to strengthen the representation of the follower's subproblem \eqref{bilevel:cons:1} in the leader's model. 
For scenarios where the follower is a linear programming, \cite{bard1982linear} introduced the now standard  dualize-and-combine method, which reformulates \eqref{bilevel:cons:1} via duality and related Karush-Kuhn-Tucker (KKT) optimality conditions. Generalizations of this approach have been of theoretical interest, including, e.g., the sequential optimality conditions by \cite{luo1996mathematical}, convex lower-level encoding by \cite{dutta2006bilevel}, and semi-definite approximations for polynomial bilevel problems by \cite{nie2017bilevel}. Dualize-and-combine, however, often requires convexity and special follower's problem structure to be computationally effective. In this work, we circumvent this requirement by explicitly enumerating the follower's subproblem values as a function of the leader's variables, exploiting symmetry and approximations to reduce the size of the underlying representation. 

A primary focus of our contribution is in providing stronger and complementary reformulations that can also be leveraged by existing MIBLP techniques. Methods based on mathematical programming have been central in both novel theoretical developments and computationally scalable procedures for bilevel optimization. \cite{DeNegreRalphs09} proposed the first branch-and-cut approach which uses the high-point relaxation as a starting point, and has now evolved to one of the state-of-the-art MibS solver \citep{tahernejad2020branch}. Recent bilevel research generalizes classical cuts to strengthen to their bilevel variants, such as the disjunctive cuts by \cite{audet2007disjunctive} and strong-duality cuts by \cite{kleinert2021closing}. Another seminal stream of works are by \cite{fischetti2016intersection,fischetti2017new}, who develop a high-performing MIBLP solver by leveraging a large class of intersection cuts and other families of inequalities. We refer to the survey by \cite{kleinert2021survey} for other related cuts and mathematical programming techniques. In this paper, we provide an approximate reformulation of the follower's subproblem \eqref{bilevel:cons:1} that projects the graph of the value function $\val$ onto the leader's decision. Such a reformulation is linear and of parameterized size, and can be incorporated in any branching search that uses existing cuts.

The proposed approximation is generated via a decision-diagram encoding of the value function $\val$. In the context of optimization, decision diagrams are network representations of the state-transition graph of a dynamic program; we refer to \cite{bergman2016,castro2022decision} for concepts and a survey of applications. Specifically, decision diagrams have been applied to represent or approximate the set of feasible solutions of a combinatorial optimization problem via a network-flow model. For discrete bilevel and general two-stage problems, \citet{cspDD}, \citet{Sebas}, and \citet{DDCarvalho} exploit this reformulation to rewrite the follower's problem as a shortest-path linear program, which now becomes amenable to a dualize-and-combine approach under certain classes of interaction constraints, such as interdiction inequalities where the leader ``blocks'' arcs of the resulting shortest-path formulation.

This work assumes no underlying structure of the follower's subproblem except for general linearity of interactions; e.g., the follower's variables can be continuous. To this end, the proposed flow model represents the domain and evaluations of $\val$ instead of feasible solutions, i.e., the network captures functional values as opposed to feasibility space. To our knowledge, the closest related work is by \cite{bergman2018discrete}, which uses state aggregation to approximate structured sums of non-linear functions (e.g., submodular sums). In our setting, ``states'' capture both the slacks of the interaction constraints and symmetry associated with follower's subproblem values. \cite{brotcorne2013one} similarly proposes a notion akin to a state to provide smaller single-level reformulations when the follower's problem is a knapsack, specifically by considering weight intervals that lead to the same follower's solution values. The state-based model used as a starting-point for our derivations in \S\ref{subsec:statebased} can be perceived as a generalization of this concept for general bilevel programs, where we use a network-flow model to extract a convex-hull model of the graph of $\val$.


\section{Extended Value Function Formulations}
\label{sec:extension}

Let $\extval \subseteq \mathbb{R}^{\nld+1}$ be the graph of the function $\val$, that is, the set of leader's decisions extended with their follower's values whenever \eqref{bilevel:cons:1} is feasible:
\begin{align*}
	\extval := \{ \, (\xvec, \z) \in \{0,1\}^{\nld} \times \mathbb{R}\, \colon \, \z = \val(\xvec), \, \z < +\infty \}.
\end{align*}

We operate on the following reformulation of the value function model \eqref{model:vfr}, which imposes that both the vector $\xvec$ and an upper bound on its follower's value are in the graph $\extval$.
\begin{subequations}
	\label{model:vfr_ext}
	\begin{align}
		\min_{\xvec, \yvec, \z}  
		&\quad 
			f(\xvec, \yvec) \label{vfr_ext:obj} \\
		\textnormal{s.t.}
		&\quad
			\textnormal{\eqref{vfr:cons:1}-\eqref{vfr:cons:2}}, 
				\label{vfr_ext:cons:1} \\
		&\quad
			g(\yvec) \le \z, 
				\label{vfr_ext:cons:2} \\
		&\quad
			(\xvec, \z) \in \extval.
				\label{vfr_ext:cons:3}
	\end{align}
\end{subequations}

Our objective in this work is to investigate extended polyhedral formulations to the set $\extval$ to solve or approximate the extension \eqref{model:vfr_ext} efficiently, with a focus on strong mathematical programming formulations in terms of their linear relaxation. Of special relevance to our approach are formulations that capture symmetric classes of the leader's decision with respect to $\val$, i.e., where a solution in a class leads to the same follower's objective value as other solutions in that class, leading to more compact models and approximations. 

\subsection{State-based model}
\label{subsec:statebased}

An explicit strategy to expose symmetry in the leader's variable set is to define
\begin{align*}
	\vale(\state) := \min_{\yvec} 
	\left\{ 
		g(\yvec) 
		\,\colon\,
		\Bf \yvec \ge \rhsf - \state, \;\; \yvec \in \yvecset
	\right\},
\end{align*}
which is the alternative value function $\vale \colon \mathbb{R}^{\ncons} \rightarrow \mathbb{R} \cup \{+\infty\}$ written in terms of the right-hand side of the interaction constraints. It follows that $\val(\xvec') = \vale(\Al \xvec')$ for any leader's decision $\xvec' \in \{0,1\}^{\nld}$; thus, to solve the extension \eqref{model:vfr_ext}, it suffices to evaluate $\vale$ only on the set of vectors generated via the linear operator $\Al$, i.e., 
\begin{align}
	\label{eq:statedef}
	\stateset := \{\state \in \mathbb{R}^{\ncons} \colon \state = \Al \xvec, \, \xvec \in \{0,1\}^{\nld}, \, \vale(\state) < +\infty\}.
\end{align} 
Each $\state \in \stateset$ can be perceived as a ``state'' in that vectors $\xvec$ that yield the same $\state$ also lead to the same follower's objective. Since $\stateset$ is finite, one can represent $\extval$ via standard indicator variables $\issvec$ capturing the active state $\state$ at a feasible solution. Namely,
\begin{align*}
	\extval^{\issvec} 
	:= 
	\left \{ 
		(\xvec, \z) \in \{0,1\}^{\nld} \times \mathbb{R}
		\, \colon\, \exists \,\, \issvec \,\, \textnormal{s.t.} 
		\quad
			\def\arraystretch{1.4}
			\begin{array}{|ll}
				\quad \Al \xvec = \sum_{\state \in \stateset} \state \cdot \iss_{\state}, \\ 
				\quad \sum_{\state \in \stateset} \iss_{\state} = 1, \\
				\quad \iss_{\state} \in \{0,1\}, \,\,\,\, \forall \state \in \stateset, \\
				\quad z = \sum_{\state \in \stateset} \vale(\state) \cdot \iss_{\state}
			\end{array}
	\right \}.
\end{align*}
In the representation above, $\iss_{\state} = 1$ if and only if $\Al \xvec = \state$ for exactly one state $\state \in \stateset$, as prescribed by the first, second, and third constraints. From the last constraint, $\extval^{\iss}$ is a valid representation of $\extval$ since $\z = \sum_{\state \in \fstateset \colon \state = \Al \xvec'} \vale(\state) \cdot \iss_{\state} = \vale(\Al \xvec') = \val(\xvec')$ for any $\xvec' \in \{0,1\}^{\nld}$.

Under the MIBLP Assumption \ref{as:binarylinearinteger}, the extension \eqref{model:vfr_ext} written in terms of $\extval^{\issvec}$ is a single-level mixed-integer linear program with an additional $\ncons+2$ constraints and $|\stateset|$ binary variables with respect to \eqref{model:vfr}. The cardinality of $\stateset$ is bounded above by $|\extval|$ because each $\xvec$ can only map to a single state $\state \in \stateset$. In other words, the resulting size of \eqref{model:vfr_ext} written in terms of $\extval^{\iss}$ -- that is, its number of variables plus constraints -- would never be worse than the analogous model obtained when enumerating values $\val(\xvec)$ for each vector $\xvec$ instead. 

In structured settings, $\stateset$ can be significantly smaller than $\extval$. A common case is when the follower's subproblem models a classical knapsack system,
\begin{align}
	\label{model:fknapsack}
	\val^{K}(\xvec) := \max_{\yvec} 
	\left\{ 
		\cvec^{\top} \yvec 
		\,\colon\,
		\dvec^{\top} \yvec \le b - \al^{\top} \xvec, \;\; \yvec \in \{0,1\}^{\nf}
	\right\},
\end{align}
where $\al \in \mathbb{R}^{\nld}$ and $\cvec, \dvec \in \mathbb{R}^{\nf}$ are positive integer vectors, and $b$ is a positive integer scalar. Such a bilevel problem formulates settings where the leader wishes to limit the follower's budget, with applications in resource allocation and scheduling \citep{brotcorne2013one, ghatkar2023solution, carvalho2023bilevel}. In this formulation, the state set $\stateset$ models the distinct possible values of $\al^{\top} \xvec$ and is such that $\stateset \subseteq \{0, 1, \dots, b \}$, as all input parameters are positive integers and $\ncons = 1$. Thus, $\stateset$ and $\extval^{\issvec}$ are pseudo-polynomial in the right-hand side $b$ irrespective of the dimension of $\xvec$. This implies that \eqref{model:vfr_ext} written in terms of $\extval^{\iss}$ is a compact single-level reformulation of the original bilevel problem \eqref{model:bilevel} when $b \in \mathcal{O}(\nld)$. 

Nonetheless, even when the number of states is small, one of the challenges of the indicator-based formulation $\extval^{\issvec}$ is that the resulting linear programming relaxation of \eqref{model:vfr_ext} can be arbitrarily weak. We illustrate this in Example \ref{ex:continousexact} for a continuous follower's subproblem.
\begin{example}
	\label{ex:continousexact}
	Suppose $\nld = 2$ and consider the follower's subproblem			
	\begin{align*}			
		\val(\xvec) := \min_{\yvec} 
		\left \{
			- \y_2 
			\colon
				y_1 = 5 \x_1 + 5 \x_2, \, y_2 \le 20\y_1, \, y_2 \le 199 - 19.8y_1
		\right \}.
	\end{align*}	
Since $\Al \xvec = -5x_1 -5x_2$ and $\xvec \in \{0,1\}^2$, the state set is $\stateset = \{0, -5, -10\}$ with $\vale(0) = 0$, 
$\vale(-5) = -100$, and $\vale(-10) = -1$. Thus, $\extval^{\issvec}$ is defined by the system
\begin{align*}
	&5x_1 + 5x_2 = 0\,\iss_{0} + 5\,\iss_{5} + 10 \, \iss_{10}, \\ 
	&\iss_{0} + \iss_{5} + \iss_{10}  = 1, \\
	&z = 0 \, \iss_{0} + -100 \iss_{5} -1 \iss_{10}.
\end{align*}
For $(x_1,x_2) = (0, 1)$, we must have $\iss_{5} = 1$ and $z = -100$. However, note that $\iss_{0} = 0.5$, $\iss_{10} = 0.5$, and $z = -0.5$ is feasible for the continuous relaxation obtained by dropping the integrality constraints on $\issvec$ and imposing $\iss_{0}, \iss_{5}, \iss_{10} \in [0,1]$. Thus, the value function reformulation for an instance of \eqref{model:vfr_ext} defined by
\begin{align*}
	\min_{\xvec \in \{0,1\}^2, \yvec, \z}  
	&\quad 
		x_1 + y_2 \\
	\textnormal{s.t.}
	&\quad
		x_1 + x_2 = 1, \, x_2 \ge 1, \\
	&\quad
		-y_2 \le \z, \\
	&\quad
		(\xvec, \z) \in \extval^{\iss},
\end{align*}
has a linear relaxation value of $0.5$ in contrast to its optimal value of $100$. We can replace $\vale(-5)$ by any negative number (with small adjustments to the constraints of $\vale(\state)$) to obtain any arbitrarily large gap.
\hfill $\square$
\end{example}


\section{Network Representation of Value Functions}
\label{sec:value}

In this section, we introduce a network model of the value function $\val$ that conceptually generalizes the state-based model from \S \ref{subsec:statebased} and will serve as the basis for our mathematical programs for the extension \eqref{model:vfr_ext}. Specifically, we define our network representation in \S\ref{subsec:valdd}, and show in \S\ref{subsec:extformulation} that any such a representation provides a network-flow linear program that characterizes the convex hull of $\extval$, that is, a polyhedral model such that its extreme points have a one-to-one mapping with the points $(\xvec, \val(\xvec)) \in \extval$. 

\subsection{Value Networks}
\label{subsec:valdd}

Let $\indexset := \{1, 2, \dots, \nld\}$ be the set of leader's variable indices, and define $\augindexset := \indexset \cup \{\nld+1\}$. We design a network model for $\extval$ in the form of a decision diagram that maps the domain and evaluations of $\val$ through its paths and terminal values, as follows.

\begin{definition}{\textsc{(Value Networks.)}}
	\label{def:valuedd}
	A value network $\net = (\nodeset, \edgeset, \edgevalvec)$ for $\val$ is a directed acyclic multi-terminal graph with node set $\nodeset$, edge set $\edgeset$, and terminal values $\edgevalvec$ satisfying properties (A)-(C) below.

	\begin{enumerate}[label=(\Alph*)]
		\item The node set $\nodeset$ is the union of $\nld+1$ disjoint subsets $\nodeset_1, \nodeset_2, \dots, \nodeset_{\nld+1}$, where $\nodeset_{1} = \{\initnode\}$ for a fixed root node $\initnode$ and each $\node \in \nodeset_{\nld+1}$ is referred to as a terminal node.
		\item An edge $\edge = (\node, \node', \x) \in \edgeset$ connects node $\node \in \nodeset_{j}$ and node $\node' \in \nodeset_{j+1}$ for some $j \in \indexset$ and is associated with a binary label $\x \in \{0,1\}$. No two outgoing edges of a node $u$ share the same label, i.e., $\x \neq \x'$ for any node $u \in \nodeset_{j}$ and pair $(\node, \node', \x), (\node, \node'', \x') \in \edgeset$, $j \in \indexset$. 
		\item Each terminal node $\node \in \nodeset_{\nld+1}$ is associated with a value $\edgeval_{\node} \in \mathbb{R}$. In particular, $(\xvec, \val(\xvec)) \in \extval$ if and only if there exists a path $(\node_1, \node_2, \x_1), (\node_2, \node_3, \x_2), \dots, (\node_{\nld}, \node_{\nld+1}, \x_{\nld})$ such that $\node_{j} \in \nodeset_{j}$ for all $j \in \augindexset$ and $\edgeval_{\node_{\nld+1}} = \val(\xvec)$. In other words, a path from the root to a terminal node maps to a leader's decision $\xvec = (\x_{1}, \dots, \x_{\nld})$ and corresponds to a terminal value equal to $\val(\xvec)$, and vice-versa. \hfill $\square$ 
	\end{enumerate}
\end{definition}

Given a value network $\net$, the follower's value $\val(\xvec)$ of a decision vector $\xvec$ is  obtained by following the unique root-terminal path with labels $\x_1, \x_2, \dots, \x_{\nld}$ and extracting the associated terminal value $\edgeval_{\node}$. Example \ref{ex:knapsack} depicts an example of such construction, also illustrating that $\val$ may admit multiple value networks satisfying Definition \ref{def:valuedd}.

\begin{example}
	\label{ex:knapsack}
	Suppose $\nld = 3$ and consider the follower's subproblem
	\begin{subequations}
		\label{eq:multiknapsackexample}
		\begin{align}			
			\val(\xvec) := \min_{\yvec \in \{0,1\}^2} 
				&\quad 
					-5 \y_1 + 3\y_2 \\
			\textnormal{s.t.} 
				&\quad 
					-3\y_1 - \,\,\,\y_2 \ge -5 + \x_1 + \x_2 + \x_3, \\
				&\quad
					-4\y_1 + 2\y_2 \ge -4 + 2\x_3.
		\end{align}
	\end{subequations}

	Figure \ref{fig:knap_dd} depicts a value networks for $\val$ above, where the $j$-th layer of the graph represents the node set $\nodeset_{j}$, and solid and dashed edges associate with labels $\x_{j} = 0$ and $\x_{j} = 1$, respectively. For example, suppose $\xvec' = (\x'_1, \x'_2, \x'_3) = (1, 1, 0)$; this corresponds to the subproblem
	$$
		\val((1,1,0)) = \min_{\yvec \in \{0,1\}^2} \left\{ -5 \y_1 + 3\y_2 \colon -3\y_1 - \y_2 \ge -3, \, -4\y_1 + 2\y_2 \ge -4 \right\}
	$$
	which has an optimal solution $\yvec^* = (1, 0)$ with value $-5$. This solution maps to the path $( (\initnode, u_{12}, 1), (u_{12}, u_{23}, 1), (u_{23}, u_{35}, 0))$ in the network. Figures \ref{fig:knap_reduction}-(a) and \ref{fig:knap_reduction}(b) also present two alternative value networks for $\val$; in particular, note that the vector $(1,1,0)$ maps to the path $( (\initnode, u_{12}, 1), (u_{12}, v_{21}, 1), (v_{22}, v_{31}, 0))$ in Figure \ref{fig:knap_reduction}-(b), ending in a terminal node with a value of $-5$. \hfill $\blacksquare$

    \begin{figure}[t]
        \begin{center}

            \tikzstyle{main node} = [circle,fill=gray!50,font=\scriptsize, inner sep=1pt]            
            \tikzstyle{text node} = [font=\scriptsize]
            \tikzstyle{arc text} = [font=\scriptsize]
			\tikzstyle{zero arc} = [draw,dashed, line width=0.5pt,->]
			\tikzstyle{one arc} = [draw,line width=0.5pt,->]
					
            \begin{tikzpicture}[->,>=stealth',shorten >=1pt,auto,node distance=1cm,
                thick]        
                \tikzstyle{node label} = [circle,fill=gray!20,font=\scriptsize, inner sep=2pt]


    \node[node label, label={[xshift=0.35em, yshift=0.5em, font=\scriptsize] left:$(0,0)$}] (r) at (-1.5375,0) {$\;\initnode\;$};

    \node[node label, label={[xshift=0.35em, yshift=0.5em, font=\scriptsize] left:$(0,0)$}] (u11)  at (-3.475,-1.25)  {$u_{11}$};
                
    \node[node label, label={[xshift=-0.35em, yshift=0.5em, font=\scriptsize] right:$(-1,0)$}] (u12)  at (0.4,-1.25)  {$u_{12}$};


    \node[node label, label={[xshift=0.35em, yshift=0.5em, font=\scriptsize] left:$(0,0)$}] (u21)  at (-5.5,-2.5) {$u_{21}$};
    
    \node[node label, label={[xshift=0.35em, yshift=0em, font=\scriptsize] left:$(-1,0)$}] (u22)  at (-1.5375,-2.5) {$u_{22}$};
        
    \node[node label, label={[xshift=-0.35em, yshift=0.5em, font=\scriptsize] right:$(-2,0)$}] (u23)  at (2.25,-2.5) {$u_{23}$};


    \node[node label, label={[xshift=0.35em, yshift=0.5em, font=\scriptsize] left:$(0,0)$}, label={[color=blue!50!black] south:$-5$}] (u31)  at (-6.5,-3.9) {$u_{31}$};

    \node[node label, label={[xshift=0.35em, yshift=0.5em, font=\scriptsize] left:$(0,-2)$}, label={[color=blue!50!black] south:$-2$}] (u32)  at (-4.5,-3.9) {$u_{32}$};

    \node[node label, label={[xshift=0.35em, yshift=0.5em, font=\scriptsize] left:$(-1,0)$}, label={[color=blue!50!black] south:$-5$}] (u33)  at (-2.5,-3.9) {$u_{33}$};

    \node[node label, label={[xshift=0.35em, yshift=0.5em, font=\scriptsize] left:$(-2,-2)$}, label={[color=blue!50!black] south:$0$}] (u34)  at (-0.4,-3.9) {$u_{34}$};				

    \node[node label, label={[xshift=0.35em, yshift=0.5em, font=\scriptsize] left:$(-2,0)$}, label={[color=blue!50!black] south:$-5$}] (u35)  at (1.5,-3.9) {$u_{35}$};				

    \node[node label, label={[xshift=-0.35em, yshift=0.5em, font=\scriptsize] right:$(-3,-2)$}, label={[color=blue!50!black] south:$0$}] (u36)  at (3,-3.9) {$u_{36}$};				
;

    \path[every node/.style={font=\sffamily\small}]
    (r)                     
        edge[zero arc] (u11)
        edge[one arc] (u12)
    (u11)
        edge[zero arc] (u21)
        edge[one arc] (u22)
    (u12)
        edge[zero arc] (u22)
        edge[one arc] (u23)	
    (u21)
        edge[zero arc] (u31)
        edge[one arc] (u32)
    (u22)
        edge[zero arc] (u33)
        edge[one arc] (u34)
    (u23)
        edge[zero arc] (u35)
        edge[one arc] (u36)	
    ;

				\node[font=\scriptsize] (x1) at (-7.5, -0.35) {$\x_1$} ;
				\node[font=\scriptsize] (x2) at (-7.5, -1.75) {$\x_2$} ;
				\node[font=\scriptsize] (x3) at (-7.5, -3) {$\x_3$} ;

			\end{tikzpicture}        
		\end{center}
    \caption{A value networks for the follower's subproblem \eqref{eq:multiknapsackexample}. Solid and dashed arcs at the $j$-th layer of the figure correspond to edges with labels $\x_{j} = 0$ and $\x_{j} = 1$, respectively, and numbers below the last-layer nodes are the terminal values. The network is a state-based network with states depicted on each node (\S\ref{subsec:statebasednet}).}
    \label{fig:knap_dd} 
\end{figure}
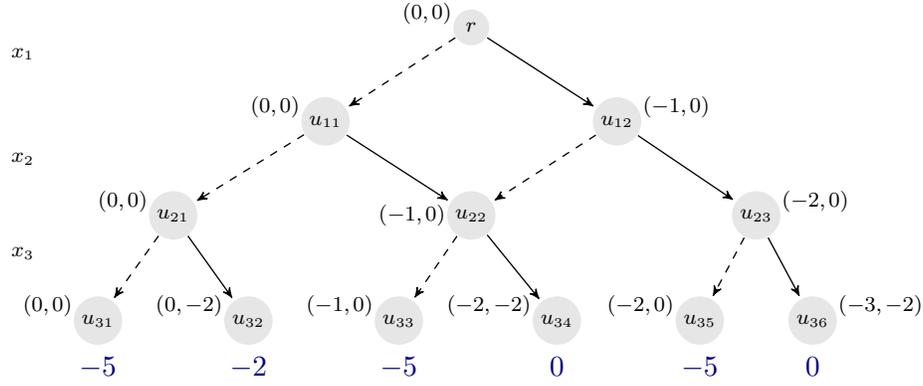
			
\end{example}

In comparison to the state-based model in \S\ref{subsec:statebased}, each terminal node $\node \in \nodeset_{\nld+1}$ defines an equivalence class in terms of its graphical structure. More precisely, all leader's decision vectors $\xvec$ associated with the labels of the paths starting at $\initnode$ and ending at a terminal node $\node$ have the same value $\val(\xvec) = \edgeval_{\node}$. The terminal nodes in Figure \ref{fig:knap_dd}, in particular, have a one-to-one mapping to states $\stateset$. 

Intermediate nodes $\node \in \nodeset \setminus \nodeset_{\nld+1}$, in turn, admit a similar but slightly more general symmetry interpretation, in that any two vectors $\xvec_1, \xvec_2$ extracted from an $(\initnode,\node')$-path label can be \textit{completed} in the same way. That is, for any $\xvec'$ associated with a partial path starting from $\node$ and ending at a terminal $\node' \in \nodeset_{\nld+1}$, 
the solutions $(\xvec_1, \xvec')$ and $(\xvec_2, \xvec')$ have the same follower's value $\edgeval_{\node'}$. For example, in Figure \ref{fig:knap_dd}, all paths ending at node $u_{22}$ have two completions, $\x_{3} = 0$ and $\x_{3} = 1$. Regardless how $u_{22}$ is reached from $\initnode$, all resulting leader's solutions would yield the same follower's value for either $x_3 = 0$ or $x_3 = 1$.

The alternative representations in Figure \ref{fig:knap_reduction} compress the network size by removing all redundancies related to the symmetry of both terminal and intermediate notes discussed above. We discuss network construction procedures and symmetry reduction in Section \ref{sec:networkconstruction}.

\subsection{Convex-hull Formulation}
\label{subsec:extformulation}

We next write a convex-hull formulation of $\extval$ through an edge-based flow model of $\net$ with additional linking constraints for $\xvec$ and $\z$. More precisely, for any node $\node \in \nodeset$, let $\outedges(\node) := \{ (\node, \node', \x) \in \edgeset \colon \node' \in \nodeset, \x \in \{0,1\} \}$ and $\inedges(\node) := \{ (\node', \node, \x) \in \edgeset, \colon \node' \in \nodeset, \x \in \{0,1\} \}$ be the set of outgoing and incoming edges at a $\node$, respectively. We present our formulation in Proposition \ref{prop:convexhull}.

\begin{proposition}{\textsc{(Convex-hull Formulation of the Graph $\extval$.)}}
	\label{prop:convexhull}
	Given a value network $\net$, let $\polynet(\net)$ be the polytope defined by
	\begin{subequations}
		\label{model:extform}
		\begin{align}
		\polynet(\net) := \bigg \{ (\xvec, \z) 
		\colon \exists \flowvec \in \mathbb{R}^{|\edgeset|} \,\textnormal{s.t.}\,
		&\sum_{\edge \in \outedges(\initnode)} \flow_{\edge} = 1, 
				\label{extform:cons:1} \\	
		&\sum_{\edge \in \outedges(\node)} \flow_{\edge} - \sum_{\edge \in \inedges(\node)} \flow_{\edge}
		= 0, &&\forall j \in \augindexset\setminus\{1\},  \forall \node \in \nodeset_{j}, 
			\label{extform:cons:2} \\
		&\sum_{ \substack{ \edge = (\node, \node', 1) \colon \node \in \nodeset_{j}}} \flow_{\edge} = \x_{j}, &&\forall j \in \indexset, 
			\label{extform:cons:3} \\
		&z - \sum_{\node \in \nodeset_{\nld+1}} \sum_{\edge \in \edgeset^{-}(\node)} \edgeval_{\node} \, \flow_{\edge} = 0,
				\label{extform:cons:5} \\	
		&\flowvec \ge 0 \,\,\, \}.
			\label{extform:cons:6}
		\end{align}
	\end{subequations}
	Then, $(\xvec^*, \z^*)$ is an extreme point of $\polynet(\net)$ if and only if $(\xvec^*, \z^*) \in \extval$, i.e., $\polynet(\net) = \textnormal{conv}(\extval)$.
\end{proposition}

\medskip
The continuous model $\polynet(\net)$ represented by \eqref{model:extform} is composed of $|\edgeset|$ variables and $|\nodeset|+|\indexset|+1$ constraints. We can replace \eqref{vfr_ext:cons:3} by $\polynet(\net_{A})$ in the extension \eqref{model:vfr_ext} to obtain an alternative single-level exact reformulation of the bilevel program. Due to Proposition \ref{prop:convexhull}, the continuous relaxation would be either the same or stronger than the one obtained when using the state-based formulation $\extval^{\issvec}$; e.g., in Example \ref{ex:continousexact}, the resulting linear programming relaxation from $\polynet(\net)$ is optimal. The trade-off is the additional number of variables and constraints required for the reformulation, further discussed in the size analysis of \S\ref{sec:networkconstruction}.


\section{Network Construction and Symmetry-based Reduction}
\label{sec:networkconstruction}

This section provides a general procedure to build $\net$ and identify symmetries that are exposed by the underlying network structure. In \S \ref{subsec:statebasednet}, we show how to construct an initial value network by expanding states $\stateset$ from the representation in \S\ref{subsec:statebased} component-wise. Next, in \S \ref{subsec:reduction} we describe a compression algorithm that removes the redundancies associated with partial paths that evaluate to the same $\val$.




\subsection{State-based network}
\label{subsec:statebasednet}

Let $\al_{1}, \dots, \al_{j}$ denote the columns of $\Al$. Recall that the state-based representation in \S\ref{subsec:statebased} enumerates the states 
$\stateset := \{\state \in \mathbb{R}^{\ncons} \colon \state = \Al \xvec, \, \xvec \in \{0,1\}^{\nld}, \, \vale(\state) < +\infty\}$ from \eqref{eq:statedef} that admit a feasible follower's subproblem. For our subsequent derivation, define the state sets $\stateset_1, \dots, \stateset_{\nld+1}$ by setting
$\stateset_{\nld+1} := \stateset$ and 
\begin{align}
	\label{eq:partialstateset}
	\stateset_{j} := \{\state \in \mathbb{R}^{\ncons} \, \colon \, \state = \state' - \al_{j} \x, \, \forall \state' \in \stateset_{j+1}, \, \forall \x \in \{0,1\} \}, \,\,\, \forall j \in \indexset.
\end{align}
In other words, each element $\state \in \stateset_{j}$ is the partial sum $\state = \al_{1} \x_{1} + \al_{2} \x_{2} + \dots + \al_{j-1} \x_{j-1} \in \mathbb{R}^{\ncons}$ for some $(\x_1, \dots, \x_{j-1})$ that can be completed to some state in $\stateset$. Moreover, note that $\stateset_{1} = \{ \zeros \}$ for the zero vector $\zeros \in \mathbb{R}^{\ncons}$. We have that $\stateset_{1}, \dots, \stateset_{\nld+1}$ provide a value network for $\val$ by mapping states to nodes and state transitions to edges, as more rigorously defined in Lemma \ref{lem:statevalue}.  

\begin{lemma}{\textsc{(State-based Networks)}.}
	\label{lem:statevalue}
	Let $\net^{S} = (\nodeset^S, \edgeset^S, \edgevalvec^S)$ be such that
	\begin{enumerate}[label=(\alph*)]
		\item $\nodeset^S_{j} = \stateset_{j}$ for all $j \in \augindexset$, i.e., each node $\node \in \nodeset^S_{j}$ is a state $\state \in \stateset_{j}$.
		\item There exists an edge $(\state, \state', \x) \in \edgeset^S$ if and only if $\state \in \nodeset^S_{j}$, $\state' \in \nodeset^S_{j+1}$, $\x \in \{0,1\}$, and $\state' = \state + \al_{j} \x$ for some $j \in \indexset$.
		\item $\edgeval^S_{\state} = \vale(\state)$ for all $\state \in \nodeset^S_{\nld+1} = \stateset_{\nld+1}$.
	\end{enumerate}

	\smallskip
	Then, $\net^{S}$ is a value network for $\val$.
\end{lemma}

We refer to $\net^{S}$ as the state-based network of $\val$. In principle, if the problem input size is small (e.g., a knapsack case as presented in \S\ref{subsec:statebased}), $\net^{S}$ can be constructed via the following three steps:
\begin{enumerate}
	\item Enumerate the states $\stateset_{\nld+1}$ by listing all vectors $\{ \Al \xvec \colon \xvec \in \{0,1\}^{\nld} \}$ and removing infeasible states.
	\item Obtain $\stateset_{\nld}, \stateset_{\nld-1}, \dots, \stateset_{1}$, in order, via the relation \eqref{eq:partialstateset}. 
        \item Set $\nodeset^S_1, \dots, \nodeset^S_{\nld+1}$ as $\stateset_1, \dots, \stateset_{\nld+1}$, respectively, and create $\edgeset^S$ reflecting the state transitions of
        \eqref{eq:partialstateset}.
\end{enumerate}

\begin{example}
	\label{ex:statenetwork}
Figure \ref{fig:knap_dd}-(a) depicts the state-based network $\net^{S}$ for the value function $\val$ from Example \ref{ex:knapsack}. In particular, given the $\ncons=2$ interaction constraints, a terminal state $\statevec \in  \stateset_{\nld+1} = \nodeset^S_{\nld+1} \subseteq \mathbb{R}^2$ is described by
\begin{align*}
    \statevec = \Al \xvec =
    \left[
        \begin{array}{ccc}
        -1 & -1 & -1 \\
        0  & 0 & -2 \\
        \end{array}
    \right]
    \cdot 
    \left [
        \begin{array}{c}
        \x_1 \\
        \x_2 \\
        \x_3 
        \end{array}
    \right ]
    = 
    \left [
        \begin{array}{c}
        -1 \\
        0 \\
        \end{array}
    \right ]    
    \x_1
    +
    \left [
        \begin{array}{c}
        -1 \\
        0 \\
        \end{array}
    \right ]    
    \x_2
    +
    \left [
        \begin{array}{c}
        -1 \\
        -2 \\
        \end{array}
    \right ]
    \x_3
\end{align*}
for some $\xvec = (\x_1, \x_2, \x_3) \in \{0,1\}^{\nld}$. Analogously, any state in layer $j \ge 2$ can be described by considering the product of the first $j-1$ columns of $\Al$ and $j-1$ components of $\xvec$. The labels above a node at the $j$-th layer represent the associated state $\statevec \in \stateset_{j} \in \mathbb{R}^2$, $j \in \augindexset$. In particular, $(\x_1, \x_2, \x_3) = (1, 1, 0)$ results in the state $\statevec \in \stateset_{\nld+1}$ such that $\statevec = (-\x_1 - \x_2 - \x_3, -2x_3) = (-2, 0)$, which has a value of $\vale(\statevec) = -5$. \hfill $\square$
\end{example}

\subsection{Size of state-based networks} 
\label{sec:sizestatebased}

The total number of nodes in a state-based network is $|\nodeset^{S}| = \sum_{j=1}^{\nld+1} |\stateset_{j}|$ and the number of edges is bounded by $2|\nodeset^{S}|$, since each node has at most two outgoing edges as labels are binaries. The cardinality of each $|\stateset_{j}|$ at some layer $j$ depends on the number of distinct vectors $\al_1 \x_1 + \dots + \al_{j-1} \x_{j-1} \in \mathbb{R}^{\ncons}$ that leads to a feasible state $\stateset_{\nld+1}$. In particular, for the $i$-th interaction constraint, $i \in \{1,\dots,\ncons\}$, the number of distinct elements $L_{ij}$ at a layer $j$ is given by the number of subset sums of the set $\{a_{i,1}, a_{i,2}, \dots, a_{i,j-1}\}$, i.e., $L_{ij} := \left|\left\{\sum_{j'=1}^{j-1} a_{i,j'} \x_{j'} \colon \xvec \in \{0,1\}^{\nld} \right\} \right|$. If $L := \max_{i \in \{1,\dots,\ncons\}, j \in \augindexset} L_{ij}$ is the maximum across all interaction constraints, then $\stateset_j \in \mathcal{O}(L^m)$ and $|\nodeset^{S}| \in \mathcal{O}(\nld \, L^m)$. Thus, the size of $\net^S$ will be tractable under conditions that restrict either the dimension or the variance of $\Al$; for instance,
\begin{itemize}
\item The number of interaction constraints $\ncons$ is small and $\Al$ carries structure that allows us to rule out infeasible follower's subproblems. For example, if the follower's subproblem is a knapsack constraint, $\ncons = 1$ and $\Al$ are positive integers (see \S\ref{subsec:statebased}). In such a case, $L$ is bounded by the right-hand side of the knapsack constraint, $b$, and the network $\net_S$ is pseudo-polynomial in the problem input. This implies that the single-level reformulation obtained from Proposition \ref{prop:convexhull} is also pseudo-polynomial.

\item The interaction matrix $\Al$ is sparse or contains multiple repeated elements, restricting the number of subset sums and therefore the possible values of $L$.
\end{itemize}


In particular, if $\Al \ge \zeros$, then $\stateset_j \subseteq \stateset_{\nld+1}$ for all $j \in \indexset$ because $\statevec + \zeros \in \stateset_{\nld+1}$ for any $\statevec \in \stateset_{j}$, that is, one obtains a feasible state by setting the remaining leader's components to $0$. For example, note that every node in an intermediate layer in Figure \ref{fig:knap_dd} has a path to a terminal node that only include edges with labels $\x = 0$. It follows that $|\stateset_j| \in \mathcal{O}(|\stateset_{\nld+1}|) = \mathcal{O}(|\stateset|)$, and therefore $|\nodeset^{S}| \in \mathcal{O}(\nld \cdot |\stateset|)$. That is, the trade-off in generating the convex hull of $\extval$ is, in the worst case, replicating $\stateset$ for each of the $\nld$ leader's components.


\begin{remark}[Edge lengths]
	\label{ex:remark}
	It is common in the decision diagram literature to associate arcs with lengths, as opposed to considering only node terminal values. In particular, arc lengths lead to formulations that represent more traditional shortest-path problems (see, e.g., \citealt{bergman2018discrete}). If such a perspective is desired, one could define an equivalent value network with edge lengths, for example, by distributing node terminal values to arcs via a bottom-up approach \citep{hooker2013decision}. Nonetheless, in our context, such a transformation has no direct effect in terms of the linear programming relaxation as implied by Proposition \ref{prop:convexhull}, nor in the minimality of the network size (see \S\ref{subsec:reduction}). 
\end{remark}

\subsection{Reduction via Symmetric Nodes}
\label{subsec:reduction}

The combinatorial structure of the network exposes additional symmetries of the value function that are not captured by  the linear form $\Al \xvec$ embedded in $\vale$. Specifically, if two terminal nodes have the same value, such as $u_{31}$ and $u_{33}$ in Figure \ref{fig:knap_dd}, they are effectively redundant to the representation and could be intuitively replaced by a single node. We can leverage the same interpretation from Section \ref{subsec:valdd} for intermediate nodes $\node, \node' \in \nodeset \setminus \nodeset_{\nld+1}$. Namely, if $\node$ and $\node'$ lead to completions ending at the same terminal values, they also capture the same information. This is the case, for instance, of nodes $u_{22}$ and $u_{23}$ in Figure \ref{fig:knap_dd}. We formalize this concept in Definition \ref{def:symmetryclass}.

\begin{definition}{\textsc{(Symmetric Nodes and Classes.)}}
	\label{def:symmetryclass}
	The nodes $u, v \in \nodeset_{j}$, $j \in \augindexset$, are \textit{symmetric} if for any path $(\node, \node')$ starting at $u$ and ending at a terminal node $\node' \in \nodeset_{\nld+1}$, there exists a path $(v, \node'')$ starting at $v$ and ending at $\node'' \in \nodeset_{\nld+1}$ such that 
	\begin{enumerate}
		\item[(i)]  if $j < \nld+1$, both the paths $(u, \node')$ and $(v, \node'')$ have the same labels $\xvec = (\x_{j}, \dots, \x_{\nld})$; and
		\item[(ii)] the terminal nodes $\node'$ and $\node''$ have the same values, i.e., $\edgeval_{\node'} = \edgeval_{\node''}$;
	\end{enumerate}
	and vice-versa for $v$ and $\node$.  A \textit{symmetry class} is a subset of nodes $\symgroup \subseteq \stateset_{j}$ that are pairwise symmetric.
	\hfill $\square$
\end{definition}


For example, in the state-based network of Figure \ref{fig:knap_dd}, nodes $u_{22}$ and $u_{23}$ in the previous-to-last layer are symmetric, despite their associated states being distinct. Analogously, the terminal node set $\{u_{31}, u_{33}, u_{35}\}$ also defines a symmetry class. Symmetric nodes are redundant as they represent the same set of paths and value function evaluations, and it suffices to represent only one node per symmetry class. That is, nodes in a symmetry class can be equivalently ``merged'' without loss of information, as we develop more rigorously in this section. 
Notably, even if two nodes have the same path labels to terminal nodes, they may not lead to the same follower's value; for example, nodes $u_{21}$ and $u_{23}$ in Figure \ref{fig:knap_dd} have the same completions but they are not symmetric. 
This distinction separates symmetry from the traditional concept of node equivalence from the existing decision diagram literature (see, e.g., \citealt{bergman2016}). 

Nonetheless, similar reduction algorithms can be applied to construct a value network that only contains one node per symmetry class, i.e., it is minimal in size. Further, this value network is unique (up to isomorphism). To demonstrate this, the following lemma states two sufficient conditions for nodes to be symmetric that we apply in our reduction methodology.
\begin{lemma}{(\sc{Reduction Conditions})}
	\label{lemm:reduction}
	For any two nodes $u,v \in \nodeset$ at the same layer:
	\begin{enumerate}
		\item[(R1)] If $u,v \in \nodeset_{\nld+1}$ are terminal nodes, $u$ and $v$ are symmetric if and only if $\edgeval_{u} = \edgeval_{v}$.

		\item[(R2)] If $u,v \in \nodeset \setminus \nodeset_{\nld+1}$, $u$ and $v$ are symmetric if for every edge $(u,u',\x) \in \edgeset$, there exists an edge $(v,u',\x) \in \edgeset$ directed to the same target node $u'$, and vice-versa.
	\end{enumerate}
\end{lemma}

Both conditions (R1) and (R2) are directly implied by Definition \ref{def:symmetryclass}, where (R2) indicates that outgoing edges from $u$ and $v$ with the same label target the same node in a subsequent layer; thus, they must result in the same paths to terminal nodes. 

Based on Lemma \ref{lemm:reduction}, Algorithm \ref{algo:reduction} presents a reduction procedure to compress a given value network $\net$ by only preserving one node per symmetry class. More specifically, the procedure applies condition (R1) at the terminal layer (lines 1-3) to remove value redundancy, i.e., to ensure that the last layer only contains one node per distinct value of $\val$. Condition (R2) is then applied iteratively at layers $\nld, \dots, 1$ (lines 4-7) to identify the remaining symmetric nodes, which are then appropriately removed by redirecting their incoming arcs to another node of the class (note that their outgoing nodes are already redundant). 




\begin{algorithm}	
	\begin{algorithmic}[1] 
	\Procedure{Reduce}{$\net$}
	\While{$\exists \, u,v \in \nodeset_{\nld+1}$ such that (R1) is satisfied }
		\State Redirect incoming edges at $v$ to $u$ 
		\State $\nodeset_{j} \gets \nodeset_{j} \setminus \{v\}$ 
	\EndWhile
	\For{$j=\nld, \dots, 1$}
		\While{$\exists \, u,v \in \nodeset_{j}$ such that (R2) is satisfied }
			\State Redirect incoming edges at $v$ to $u$ \;
			\State $\nodeset_{j} \gets \nodeset_{j} \setminus \{v\}$ 	
		\EndWhile
	\EndFor
	\EndProcedure
	\end{algorithmic}
 	\caption{Reduction Procedure}
	\label{algo:reduction}
\end{algorithm}

The procedure in Algorithm \ref{algo:reduction} is a network reduction algorithm augmented with the terminal value conditions. We can analogously extend existing results on minimality and uniqueness. More precisely, Proposition \ref{prop:minimality} states that Algorithm \ref{algo:reduction} retuns a  value network $\net$ with no symmetries for any arbitrary input, and that $\net$ is unique under the leader's variable indexing $\indexset$. In particular, this imply that $\net$ is necessarily minimal in terms of the number of nodes $|\nodeset|$ under this same indexing, since the presence of symmetry always imply the existence of redundant nodes in $\net$.

\begin{proposition}
	\label{prop:minimality}
	The modified value network $\net$ from Algorithm \ref{algo:reduction} is such that no two nodes $u,v \in \nodeset$ are symmetric. Moreover, $\net$ is unique, i.e., there does not exist an alternative value network that does not include symmetric nodes.
\end{proposition}



    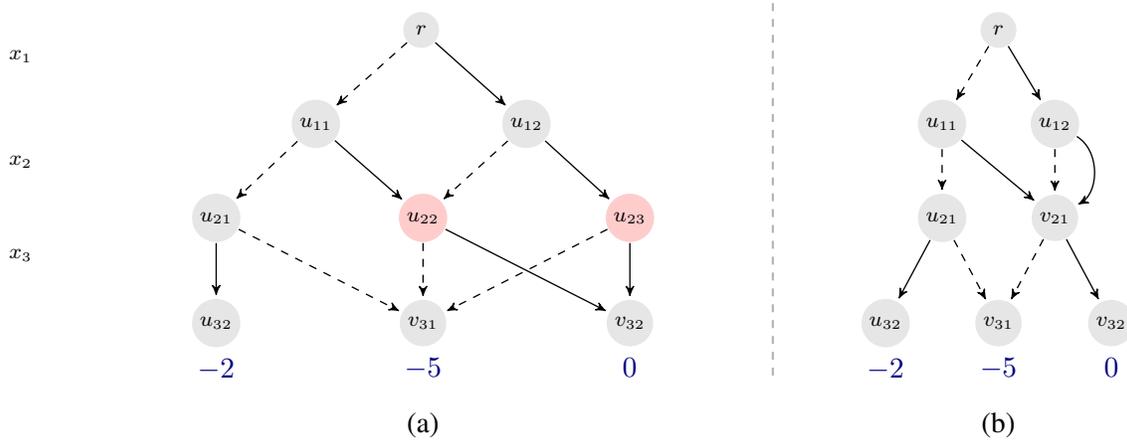
\begin{figure}[t]
        \begin{center}

            \tikzstyle{main node} = [circle,fill=gray!50,font=\scriptsize, inner sep=1pt]            
            \tikzstyle{text node} = [font=\scriptsize]
            \tikzstyle{arc text} = [font=\scriptsize]
			\tikzstyle{zero arc} = [draw,dashed, line width=0.5pt,->]
			\tikzstyle{one arc} = [draw,line width=0.5pt,->]
					
            \begin{tikzpicture}[->,>=stealth',shorten >=1pt,auto,node distance=1cm,
                thick]        
                \tikzstyle{node label} = [circle,fill=gray!20,font=\scriptsize, inner sep=2pt]


                \node[node label] (r) at (-2.175,0) {$\;\initnode\;$};

				\node[node label] (u11)  at (-3.575,-1.25)  {$u_{11}$};
                \node[node label] (u12)  at (-0.775,-1.25)  {$u_{12}$};

				\node[node label] (u21)  at (-4.9,-2.5) {$u_{21}$};
                \node[node label, fill=red!20] (u22)  at (-2.15,-2.5) {$u_{22}$};
                \node[node label, fill=red!20] (u23)  at (0.6,-2.5) {$u_{23}$};

				\node[node label, label={[color=blue!50!black] south:$-2$}] (u32)  at (-4.9,-3.9)  {$u_{32}$};
				\node[node label, label={[color=blue!50!black] south:$-5$}] (v31)  at (-2.15,-3.9) {$v_{31}$};
				\node[node label, label={[color=blue!50!black] south:$0$}] (v32)  at (0.6,-3.9) {$v_{32}$};				
				;

				\path[every node/.style={font=\sffamily\small}]
                (r)                     
                    edge[zero arc] (u11)
                    edge[one arc] (u12)
				(u11)
					edge[zero arc] (u21)
					edge[one arc] (u22)
				(u12)
					edge[zero arc] (u22)
					edge[one arc] (u23)	
				(u21)
					edge[zero arc] (v31)
					edge[one arc] (u32)
				(u22)
					edge[zero arc] (v31)
					edge[one arc] (v32)
				(u23)
					edge[zero arc] (v31)
					edge[one arc] (v32)	
				;


				\node (sepA) at (2.5, 0.5) {};
				\node (sepB) at (2.5, -4.8) {};
				\draw[dashed, -, color=black!30] (sepA) -- (sepB);


                \node[node label] (rva) at (5.5,0) {$\;\initnode\;$};

				\node[node label] (v11a) at (4.75,-1.25) {$u_{11}$};
				\node[node label] (v12a) at (6.25,-1.25) {$u_{12}$};

				\node[node label] (v21a) at (4.75,-2.5) {$u_{21}$};
				\node[node label] (v22a) at (6.25,-2.5) {$v_{21}$};

				\node[node label, label={[color=blue!50!black] south:$-2$}] (u32a) at (4,-3.9) {$u_{32}$};
				\node[node label, label={[color=blue!50!black] south:$-5$}] (v31a) at (5.5,-3.9) {$v_{31}$};
				\node[node label, label={[color=blue!50!black] south:$0$}] (v33a) at (7,-3.9) {$v_{32}$};
				;

				\path[every node/.style={font=\sffamily\small}]
				(rva)
					edge[zero arc] (v11a)
					edge[one arc] (v12a)
				(v11a)
					edge[zero arc] (v21a)
					edge[one arc] (v22a)
				(v12a)
					edge[zero arc] (v22a)
					edge[one arc, bend left=60] (v22a)
				(v21a)
					edge[one arc] (u32a)
					edge[zero arc] (v31a)
				(v22a)
					edge[zero arc] (v31a)
					edge[one arc] (v33a)
				;


				\node(item1) at (-2.15, -5.25) {(a)} ;
				\node(item2) at (5.5, -5.25) {(b)} ;


				\node[font=\scriptsize] (x1) at (-7.5, -0.35) {$\x_1$} ;
				\node[font=\scriptsize] (x2) at (-7.5, -1.75) {$\x_2$} ;
				\node[font=\scriptsize] (x3) at (-7.5, -3) {$\x_3$} ;

			\end{tikzpicture}        
		\end{center}
    \caption{Two alternative value networks for the follower's subproblem \eqref{eq:multiknapsackexample}. Coloured nodes depict a symmetry class. The networks in (a) and (b) are obtained by applying the reduction Algorithm \ref{algo:reduction} in node sets $\nodeset_{4}$ and $\nodeset_{3}$, respectively.}
    \label{fig:knap_reduction} 
\end{figure}
\begin{example}
	\label{ex:reduction}
	The networks depicted in Figures \ref{fig:knap_reduction}-(a) and \ref{fig:knap_reduction}-(b) are obtained after applying the reduction from Algorithm \ref{algo:reduction} in layers $\nodeset_{\nld+1} = \nodeset_{4}$ and $\nodeset_{\nld} = \nodeset_{3}$ to the network in Figure \ref{fig:knap_dd}, respectively. In particular, node $v_{31}$ in Figure \ref{fig:knap_reduction}-(a) aggregates the symmetric class $\{u_{31}, u_{33}, u_{35}\}$ from Figure \ref{fig:knap_dd}, and analogously $v_{32}$ aggregates the class $\{u_{34}, u_{36}\}$. Figure \ref{fig:knap_reduction}-(a) presents two symmetric nodes $\{u_{22}, u_{23}\}$ (in color), which are aggregated by node $v_{21}$ in Figure \ref{fig:knap_reduction}-(b). \hfill $\square$ 
\end{example}

\section{Network-based Approximations}
\label{sec:approximation}

In general, a value network may still be significantly large for practical purposes. We next develop an approximate value network of parameterized size that provides a relaxation of the bilevel program \eqref{model:bilevel}. The relaxation provides optimization bounds that are equal or stronger  to the high-point relaxation \eqref{vfr:obj}-\eqref{vfr:cons:2} and can used, e.g., as strengthened initial models for existing solvers. We introduce the approximation in \S\ref{subsec:avdd} and present a construction procedure based on state aggregation in \S\ref{subsec:stateagg}. Finally, in \S\ref{subsec:robust} we discuss how to further strengthen the quality of the bound via a robust optimization perspective.

\subsection{Approximate Value networks}
\label{subsec:avdd}

We define the following class of value networks to approximate the extension \eqref{model:vfr_ext}:

\begin{definition}{\textsc{(Approximate Value Networks.)}}
	\label{def:avdd}
	A network $\net_{A} = (\nodeset, \edgeset, \edgeval)$ is an approximate value network if satisfies conditions (A)-(B) of Definition \ref{def:valuedd} and
	\begin{itemize}
		\item[(D1)] (\textit{Path feasibility.}) For every solution-value pair $(\xvec, \val(\xvec)) \in \extval$, there exists an $(\initnode, \node_{\nld+1})$-path $(\node_1, \node_2, \x_1), (\node_2, \node_3, \x_2), \dots, (\node_{\nld}, \node_{\nld+1}, \x_{\nld})$ in $\net_{A}$ with labels $\xvec = (\x_1, \dots, \x_{\nld})$.

		\item[(D2)] (\textit{Objective feasibility.}) For every solution-value pair $(\xvec, \val(\xvec)) \in \extval$, the terminal node $u_{\nld+1} \in \nodeset_{\nld+1}$ of its associated $(\initnode, \node_{\nld+1})$-path in $\net_{A}$ satisfies $\edgeval_{\node_{\nld+1}} \ge \val(\xvec)$. \hfill $\square$
	\end{itemize} 
\end{definition}

Definition \ref{def:avdd} considers two approximations with respect to the original value network. In (D1), the relationship between paths in $\net_{A}$ and points in $(\xvec, \val(\xvec)) \in \extval$ is unidirectional, in that solutions $\xvec'$ extracted from paths in $\net_{A}$ may not necessarily lead to feasible follower's subproblems. Moreover, (D2) states that the terminal node value of the path representing $(\xvec, \val(\xvec)) \in \extval$ is now an \textit{upper bound} on $\val(\xvec)$, even if $\val$ solves a minimization problem. For example, any feasible follower solution $\yvec$ to $\val(\xvec)$ suffices in a valid approximation. Proposition \ref{prop:appextval} shows that these two conditions imply that solving the extension \eqref{model:vfr_ext} with $\polynet(\net_{A})$ provides a lower bound to \eqref{model:bilevel}.

\begin{proposition}
	\label{prop:appextval}
	Given an approximate value network $\net_{A} = (\nodeset, \edgeset)$, the extension \eqref{model:vfr_ext} solved by replacing \eqref{vfr_ext:cons:3} with $\polynet(\net_{A})$ is a relaxation of the original bilevel problem \eqref{model:bilevel} and provides a lower bound to its optimal solution value.
\end{proposition}
\begin{proof}{Proof of Proposition \ref{prop:appextval}.}
	Let $(\xvec^*, \yvec^*, \z^*)$ be an optimal solution of reformulation \eqref{model:vfr_ext} with respect to an exact $\extval$, i.e., $(\xvec^*, \yvec^*)$ solves \eqref{model:bilevel}. Then, 
	$\val(\xvec) < +\infty$ and, by Definition \ref{def:avdd}-(D2) and inequality \eqref{vfr_ext:cons:2}, $g(\yvec) \le \z^* \le \edgeval_{u^*}$ for the terminal node in $\nodeset$ associated with $\xvec^*$. That is, $(\xvec^*, \yvec^*, \z^*)$ is also feasible to \eqref{model:vfr_ext} when solved with respect to $\polynet(\net_{A})$, and it yields the same objective function value. \hfill $\blacksquare$
\end{proof}

\begin{example}
	\label{ex:appvdd}
        Consider the follower's subproblem below for $\xvec \in \{0,1\}^3$:
	\begin{subequations}
		\label{eq:multiknapsackexample_approximation}
		\begin{align}			
			\val(\xvec) := \min_{\yvec \in \{0,1\}^2} 
				&\quad 
					-100 \y_1 - 100\y_2 \\
			\textnormal{s.t.} 
				&\quad 
					-6\y_1 \ge -10 + \x_1 + 2\x_2 + 3\x_3, \\
				&\quad
					-7\y_2 \ge -10 + 3\x_1 + 2 \x_2 + \x_3.
		\end{align}
	\end{subequations}
        Figure \ref{fig:approx_dd}-(c) depicts an approximate value network for $\val$ above. In particular, note that $\val(\xvec) = -200$ for $\xvec = (0,1,0)$, which has an optimal follower's solution value of $\yvec=(1,1)$. The corresponding path in Figure \ref{fig:approx_dd}-(c) crosses nodes $\initnode$, $u_{11}$, $u_{21}$, and $u_{31}$, the latter providing an upper bound on the value function of $0$. Using such a network in the reformulation of \eqref{vfr_ext:cons:3} provides a lower bound to the bilevel problem. \hfill $\square$
\end{example}

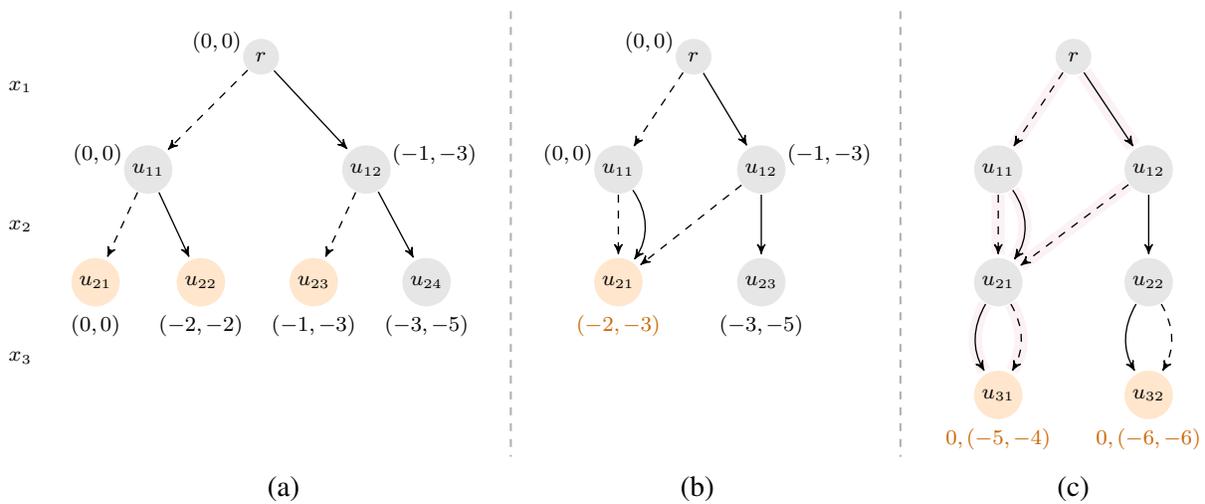
\begin{figure}[t]
    \begin{center}

        \tikzstyle{main node} = [circle,fill=gray!50,font=\scriptsize, inner sep=1pt]            
        \tikzstyle{text node} = [font=\scriptsize]
        \tikzstyle{arc text} = [font=\scriptsize]
        \tikzstyle{zero arc} = [draw,dashed, line width=0.5pt,->]
        \tikzstyle{one arc} = [draw,line width=0.5pt,->]
                
        \begin{tikzpicture}[->,>=stealth',shorten >=1pt,auto,node distance=1cm,
            thick]        
            \tikzstyle{node label} = [circle,fill=gray!20,font=\scriptsize, inner sep=2pt]


            \node[node label, label={[xshift=0.35em, yshift=0.5em, font=\scriptsize] left:$(0,0)$}] (r) at (-4.3,0) {$\;\initnode\;$};

            \node[node label, label={[xshift=0.35em, yshift=0.5em, font=\scriptsize] left:$(0,0)$}] (u11)  at (-5.8,-1.5)  {$u_{11}$};
            
            \node[node label, label={[xshift=-0.35em, yshift=0.5em, font=\scriptsize] right:$(-1,-3)$}] (u12)  at (-2.9,-1.5)  {$u_{12}$};

            \node[node label, fill=orange!20, label={[xshift=0em, yshift=-1.5em, font=\scriptsize] center:$(0,0)$}] (u21)  at (-6.5,-3) {$u_{21}$};

            \node[node label, fill=orange!20, label={[xshift=0em, yshift=-1.5em, font=\scriptsize] center:$(-2,-2)$}] (u22)  at (-5.1,-3) {$u_{22}$};

            \node[node label, fill=orange!20, label={[xshift=0em, yshift=-1.5em, font=\scriptsize] center:$(-1,-3)$}] (u23)  at (-3.6,-3) {$u_{23}$};

            \node[node label, label={[xshift=0em, yshift=-1.5em, font=\scriptsize] center:$(-3,-5)$}] (u24)  at (-2.1,-3) {$u_{24}$};

            \path[every node/.style={font=\sffamily\small}]
            (r)                     
                edge[zero arc] (u11)
                edge[one arc] (u12)
            (u11)
                edge[zero arc] (u21)
                edge[one arc] (u22)
            (u12)
                edge[zero arc] (u23)
                edge[one arc] (u24)
            ;


            \node[node label, label={[xshift=0.35em, yshift=0.5em, font=\scriptsize] left:$(0,0)$}] (r2) at (1.45,0) {$\;\initnode\;$};

            \node[node label, label={[xshift=0.35em, yshift=0.5em, font=\scriptsize] left:$(0,0)$}] (v11)  at (0.45,-1.50)  {$u_{11}$};
            \node[node label, label={[xshift=-0.35em, yshift=0.5em, font=\scriptsize] right:$(-1,-3)$}] (v12)  at (2.35,-1.50)  {$u_{12}$};

            \node[node label, fill=orange!20, label={[xshift=0em, yshift=-1.5em, font=\scriptsize, color=orange!80!black] center:$(-2,-3)$}] (v21)  at (0.45,-3) {$u_{21}$};

            \node[node label, label={[xshift=0em, yshift=-1.5em, font=\scriptsize] center:$(-3,-5)$}] (v23)  at (2.35,-3) {$u_{23}$};

            \path[every node/.style={font=\sffamily\small}]
            (r2)                     
                edge[zero arc] (v11)
                edge[one arc] (v12)
            (v11)
                edge[zero arc] (v21)
                edge[one arc, bend left=35] (v21)
            (v12)
                edge[zero arc] (v21)
                edge[one arc] (v23)
            ;


            \node[node label] (r2) at (6.5,0) {$\;\initnode\;$};

            \node[node label] (v11)  at (5.5,-1.50)  {$u_{11}$};
            \node[node label] (v12)  at (7.5,-1.50)  {$u_{12}$};

            \node[node label] (v21)  at (5.5,-3) {$u_{21}$};
            \node[node label] (v23)  at (7.5,-3) {$u_{22}$};

            \node[node label, fill=orange!20, label={[xshift=0em, yshift=-1.5em, font=\scriptsize, color=orange!80!black] center:$0, (-5,-4)$}] (v31)  at (5.5,-4.5) {$u_{31}$};
            \node[node label, fill=orange!20, label={[xshift=0em, yshift=-1.5em, font=\scriptsize, color=orange!80!black] center:$0, (-6,-6)$}] (v32)  at (7.5,-4.5) {$u_{32}$};

            \path[every node/.style={font=\sffamily\small}]
            (r2)
                edge[-, thick, line width=5, purple!5!white] (v11)
                edge[zero arc] (v11)

                edge[-, thick, line width=6, purple!5!white] (v12)
                edge[one arc] (v12)
            (v11)
                edge[-, thick, line width=5, purple!5!white] (v21)
                edge[zero arc] (v21)
                
                edge[-, thick, line width=5, purple!5!white, bend left=35] (v21)
                edge[one arc, bend left=35] (v21)
            (v12)
                edge[-, thick, line width=5, purple!5!white] (v21)
                edge[zero arc] (v21)

                edge[one arc] (v23)
            (v21)
                edge[-, thick, line width=5, purple!5!white, bend left=30] (v31)                
                edge[zero arc, bend left=30] (v31)

                edge[-, thick, line width=5, purple!5!white, bend right=30] (v31)
                edge[one arc, bend right=30] (v31)
            (v23)
                edge[zero arc, bend left=30] (v32)
                edge[one arc, bend right=30] (v32)
            ;	


            \node (sepA) at (-0.98, 0.75) {};
            \node (sepB) at (-0.98, -5.5) {};
            \draw[dashed, -, color=black!30] (sepA) -- (sepB);


            \node (sepA) at (4.2, 0.75) {};
            \node (sepB) at (4.2, -5.5) {};
            \draw[dashed, -, color=black!30] (sepA) -- (sepB);


            \node[font=\scriptsize] (x1) at (-7.5, -0.4) {$\x_1$} ;
            \node[font=\scriptsize] (x2) at (-7.5, -2.25) {$\x_2$} ;
            \node[font=\scriptsize] (x3) at (-7.5, -4) {$\x_3$} ;


            \node(item1) at (-4, -5.75) {(a)} ;
            \node(item2) at (1.5, -5.75) {(b)} ;
            \node(item3) at (6.5, -5.75) {(c)} ;

        \end{tikzpicture}        
    \end{center}
\caption{Generation steps of an approximation value network for the value function $\val$ in \eqref{eq:multiknapsackexample_approximation} (Examples \ref{ex:appvdd} and \ref{ex:approximation_example_construction}). Coloured nodes and states in Figures (b) and (c) depict aggregated/merged nodes and the smallest value of their relaxed states (i.e., $\nodelbvec$), respectively. Terminal values in (c) depict the value-function upper bound derived from such states. Shaded arcs represent the sub-network induced by the $(\initnode, \node_{31})$ paths.}
    \label{fig:approx_dd} 
\end{figure}

\subsection{Generating Approximations}
\label{subsec:stateagg}

Assume we are given a budget of $\budget$ nodes per layer previously fixed, e.g., according to memory or runtime specifications. The conditions (D1) and (D2) from Definition \ref{def:avdd} lead to a natural two-step approach to generate an approximate value network:
\begin{enumerate}
	\item Construct a network $\net$ that observes the budget $\budget$ and such that all leader's solution $\xvec$ for which $\val(\xvec) < +\infty$ are represented by some path in $\net$.
	\item For each terminal node $\node \in \nodeset_{\nld+1}$ in the resulting network, compute an upper bound $\hat{\edgeval{}}$ on the follower's values associated with the labels $\xvec$ of all $(\initnode, \node)$ paths, and set $\edgeval_{\node} = \hat{\edgeval{}}$.
\end{enumerate}

\smallskip

To construct approximations according to the two steps above, we first generalize the concept of state-based networks by reinterpreting each node as a set of states as opposed to a unique state. Specifically, recall the set of component-wise reachable states $\stateset_1, \dots, \stateset_{\nld+1}$ defined in \S\ref{subsec:statebasednet}, where $\stateset_{\nld+1} = \stateset$ for the feasible state set $\stateset$ as in \eqref{eq:statedef}. Lemma \ref{lem:statevalueapprox} is parallel to the constructive result from Lemma \ref{lem:statevalue} by considering instead a family of sets that contain $\stateset_j$, $j \in \augindexset$, as follows.

\begin{lemma}
	\label{lem:statevalueapprox}
	Let $\net_{A} = (\nodeset, \edgeset, \edgevalvec)$ be a value network satisfying 
	Definition \ref{def:valuedd}-(A),(B) where 
	\begin{enumerate}[label=(\alph*)]
		\item Each node $\Node \in \nodeset_j$, $j \in \augindexset$, is a set $\Node \subseteq \mathbb{R}^{\ncons}$ of potential states.
		\item For every $\state \in \stateset_{j}$ and $\state' \in \stateset_{j+1}$ such that $\state' = \state + \al_{j} \x$ for some $\x \in \{0,1\}$, $j \in \indexset$, there exists an edge $(\Node, \Node', \x) \in \edgeset$ where $\state \in \Node$ and $\state' \in \Node'$ for two nodes $\Node \in \nodeset_j$, $\Node' \in \nodeset_{j+1}$. 		
		\item $\edgeval_{\Node} = \max_{\state \in \Node}\vale(\state)$ for all $\Node \in \nodeset_{\nld+1}$.
	\end{enumerate}
	\smallskip
	Then, $\net_{A}$ is an approximate value network for $\val$.
\end{lemma}

Lemma \ref{lem:statevalueapprox} states that, for the purpose of constructing a valid approximate network, it suffices to have nodes as sets that cover the feasible states $\stateset_1, \dots, \stateset_{\nld+1}$, that is, \textit{state aggregations}, in addition to an edge structure that captures the corresponding state transitions. 

\subsubsection{Aggregation via Hyperrectangles.}

The result from Lemma \ref{lem:statevalueapprox} holds for any state aggregation set structure. Nonetheless, the representation and operations over such aggregated states, including the optimization in Lemma \ref{lem:statevalueapprox}-(c), must also be tractable to ensure that the network can be constructed within reasonable computational limits. 
To this end, we consider a polynomially sized representation of each node $\Node$ as a vector pair $(\nodelbvec, \nodeubvec)$ for $\nodelbvec, \nodeubvec \in \mathbb{R}^{\ncons}$, $\nodelbvec \le \nodeubvec$. In particular, $\nodelbvec$ and $\nodeubvec$ prescribe the lower bound and upper bounds, respectively, of the values of the each component $\{1,\dots,\ncons\}$ in the state set captured by $\Node$. That is, a node $(\nodelbvec, \nodeubvec)$ equivalently represents the $\ncons$-dimension hyperrectangle
$$
[\nodelb_1, \nodeub_1] \times [\nodelb_2, \nodeub_2] \times \dots [\nodelb_{\ncons}, \nodeub_{\ncons}] \subseteq \mathbb{R}^{\ncons}.
$$

We denote by $[\nodelbvec, \nodeubvec]$ the hyperrectangle implied by $\nodelbvec$ and $\nodeubvec$, as depicted above. Such a representation provides us with three properties to generate relaxations more efficiently. First, state transitions are closed under addition. That is, consider two states $\state \in \stateset_j$ and $\state' \in \stateset_{j+1}$, $j \in \indexset$, such that $\state' = \state + \al_{j} \x$ for some $\x \in \{0,1\}$ as in Lemma \ref{lem:statevalueapprox}-(b). Then,
\begin{align}
\label{eq:transitionupdate}
\state \in [\nodelbvec, \nodeubvec] \Longrightarrow \state' \in [ \nodelbvec + \al_j \x, \nodeubvec + \al_j \x ],
\end{align}
which follows by construction. In other words, if $\Node = (\nodelbvec, \nodeubvec)$, the transitions from all feasible states embedded in $\Node$ can be captured by a node $\Node' = ( \nodelbvec + \al_j \x, \nodeubvec + \al_j \x )$ and an edge $(\Node, \Node', \x)$ that links them. 

Second, we can efficiently build a superset with the same structure that contains the union of two box sets, which is key for node merging procedures. Specifically, for any pairs $(\nodelbvec_1, \nodeubvec_1)$
and $(\nodelbvec_2, \nodeubvec_2)$, let $\nodelbvec' := \left( \min \{\nodelb_{1,i}, \, \nodelb_{2,i}\} \right)_{i=1,\dots,\ncons}$ be the vector defined by the component-wise minimum of $\nodelbvec_1$ and $\nodelbvec_2$, and $\nodeubvec' := \left ( \max \{\nodeub_{1,i}, \, \nodeub_{2,i}\} \right )_{i=1,\dots,\ncons}$ be the component-wise maximum of $\nodeubvec_1$ and $\nodeubvec_2$. Then, 
\begin{align}
	\label{eq:unionupdate}
	\state \in [\nodelbvec_1, \nodeubvec_1] \cup [\nodelbvec_2, \nodeubvec_2] \Longrightarrow  \state \in [\nodelbvec', \nodeubvec'],
\end{align}
that is, $[\nodelbvec',\nodeubvec']$ preserves all states in $[\nodelbvec_1,\nodeubvec_1]$ and $[\nodelbvec_2,\nodeubvec_2]$.

Finally, for Lemma \ref{lem:statevalueapprox}-(c), the terminal node value of $\Node \in \nodeset_{\nld+1}$ is more directly calculated by 
\begin{align}
	\label{eq:optupdate}
	\edgeval_{\Node} = \max_{\state \in \Node} \vale(\state)
	= \max_{\state \in [\nodelbvec, \nodeubvec]} \vale(\state) 
	= \vale(\nodelbvec),
\end{align}
since any $\yvec \in \yvecset$ such that $\Bf \yvec \ge \rhsf - \nodelbvec$ also satisfies $\Bf \yvec \ge \rhsf - \state$ for all $\state \in [\nodelbvec, \nodeubvec]$, i.e., the follower's subproblem conditioned on states in $\Node$ is most restricted when $\state = \nodelbvec$. 

\subsubsection{Relaxation Procedure.} Building on the tractable hyperrectangle properties \eqref{eq:transitionupdate}-\eqref{eq:optupdate} above, we present the approximate network construction procedure in Algorithm \ref{algo:approxcompilation}. The methodology generates layers $\nodeset_1, \nodeset_2, \dots, \nodeset_{\nld+1}$ one at a time, starting with $\nodeset_1 = \{ (\nodelbvec_0, \nodeubvec_0) \}$ for $(\nodelbvec_0, \nodeubvec_0) = (\zeros,\zeros)$. For each layer $j \in \indexset$, three operations are performed in sequence:

\begin{enumerate}
	\item (\textit{Node generation, steps 4-6.}) For each node $\Node = (\nodelbvec, \nodeubvec) \in \nodeset_{j}$ and $\x \in \{0,1\}$, we create a new node $\Node' = (\nodelbvec + \al_{j}\x, \nodeubvec+ \al_{j}\x)$ in $\nodeset_{j+1}$ to preserve the transitions associated with $\state \in \stateset_{j} \cap [\nodelbvec, \nodeubvec]$ and $\x$ according to \eqref{eq:transitionupdate}, and add the corresponding edge to $\edgeset$.
	
	\smallskip
	\item (\textit{Node aggregation, steps 7-9.}) If the size of $\nodeset_{j+1}$ exceeds the budget $\budget$, we perform an additional state aggregation step to reduce the number of nodes in the layer. Specifically, for two selected nodes $(\nodelbvec_1, \nodeubvec_1)$
	$(\nodelbvec_2, \nodeubvec_2)$, we create a new node $[\nodelbvec',\nodeubvec']$ as in \eqref{eq:unionupdate} and remove the two first nodes, repeating the procedure until the budget $\budget$ is observed. Possible node selection strategies are discussed in \S\ref{sec:numericalstudy}. 
	
	\smallskip
	\item (\textit{Objective and Reduction, steps 10-12}.) Finally, to ensure objective feasibility, terminal node values are calculated according to \eqref{eq:optupdate} in step 11. The network is also reduced to rule out symmetric nodes.
\end{enumerate}

\begin{algorithm}[t]	
	\begin{algorithmic}[1] 
	\Procedure{GenerateApproximation}{Budget $\budget$}
	\State $\nodeset_{1} \gets (\nodelbvec_0, \nodeubvec_0)$, where $(\nodelbvec_0, \nodeubvec_0) = (\zeros, \zeros)$, and $\nodeset_{2}, \dots, \nodeset_{\nld+1}, \edgeset \gets \emptyset$
	\For{$j = 1, \dots, \nld$}		
		\For{$(\nodelbvec, \nodeubvec) \in \nodeset_{j}$ and $\x \in \{0,1\}$}		
			\Comment{\textit{Node generation steps}}		
			\State Add $(\nodelbvec', \nodeubvec') \gets (\nodelbvec + \al_{j}\x, \nodeubvec + \al_{j}\x)$ to $\nodeset_{j+1}$ (if it does not exist)
			\State Add edge $((\nodelbvec, \nodeubvec),(\nodelbvec', \nodeubvec'),\x)$ to $\edgeset$.
		\EndFor		
		\While{$|\nodeset_{j+1}| > \budget$}
			\Comment{\textit{Node merging steps.}}		
			\State Select and remove $(\nodelbvec_1, \nodeubvec_1)$, $(\nodelbvec_2, \nodeubvec_2)$ from $\nodeset_{j+1}$
			\State Add $(\nodelbvec', \nodeubvec') \gets ( \left( \min \{\nodelb_{1,i}, \, \nodelb_{2,i}\} \right)_{i=1,\dots,\ncons}, \left ( \max \{\nodeub_{1,i}, \, \nodeub_{2,i}\} \right )_{i=1,\dots,\ncons} )$
		\EndWhile
	\EndFor
	\For{$(\nodelbvec, \nodeubvec)  \in \nodeset_{\nld+1}$}
		\Comment{\textit{Calculate upper bound on terminal paths}}		
		\State $\edgeval_{(\nodelbvec, \nodeubvec)} \gets \vale(\nodelbvec)$
	\EndFor
	\State \textbf{return} \textsc{Reduce}$\left ( \net \gets (\nodeset, \edgeset, \edgeval) \right )$
	\EndProcedure
	\end{algorithmic}
 	\caption{Approximate Value network Construction}
	\label{algo:approxcompilation}
\end{algorithm}

\smallskip
The network obtained by Algorithm \ref{algo:approxcompilation} is a valid reduced approximate value network with at most $\budget$ nodes per layer. In particular, given that $\nodeset_1 = \{(\zeros,\zeros)\}$, the first layer correspond exactly to $\stateset_1$. Thus, Lemma \ref{lem:statevalueapprox}-(b) is satisfied for $\nodeset_{1}$ and $\nodeset_{2}$ (i.e., $j=1$) because of the relation \eqref{eq:transitionupdate}. The same argument can be repeated for $j=2, \dots, \nld$, given that no states are lost from node merging due to relation \eqref{eq:unionupdate}. Finally, Lemma \ref{lem:statevalueapprox}-(c) is satisfied because of \eqref{eq:optupdate}. Note also that, if the budget $\budget$ is sufficiently high, the resulting approximation matches an exact value network. 

\begin{example}
	\label{ex:approximation_example_construction}
	Figures \ref{fig:approx_dd}-(a) to (c) represent different steps of the relaxation procedure       for the value function $\val$ in \eqref{eq:multiknapsackexample_approximation} with budget $\budget = 2$. Labels at each node represent the vector $\nodelbvec$. Specifically, for a fixed $\xvec = (\x_1, \x_2, \x_{3})$, its associated state is
\begin{align*}
    \statevec = \Al \xvec =
    \left[
        \begin{array}{ccc}
        -1 & -2 & -3 \\
        -3  & -2 & -1 \\
        \end{array}
    \right]
    \cdot 
    \left [
        \begin{array}{c}
        \x_1 \\
        \x_2 \\
        \x_3 
        \end{array}
    \right ]
    =
    \left[
        \begin{array}{ccc}
        -1\\
        -3 \\
        \end{array}
    \right]
    \x_1
    +
    \left[
        \begin{array}{ccc}
        -2  \\
        -2 \\
        \end{array}
    \right]
    \x_2 
    +
    \left[
        \begin{array}{ccc}
        -3 \\
        -1 \\
        \end{array}
    \right]
    \x_3.
\end{align*}
The state updates to generate layer $j \ge 2$ are prescribed by the column vectors associated with $\x_{j-1}$ above. Figure \ref{fig:approx_dd}-(a) represents the stage at which the number of states exceeds the budget $B$. Figure \ref{fig:approx_dd}-(b) is the result of merging nodes $u_{21}$, $u_{22}$, and $u_{23}$ in steps 7-9 of Algorithm \ref{algo:approxcompilation}; in particular, the lower bound for the first and second components of the merged nodes are $\nodelb_1 = \min\{0, -2, -1\} = -2$
and $\nodelb_2 = \min\{0, -2, -3\} = -3$, respectively. Figure \ref{fig:approx_dd}-(c) is the final value network (with no reductions) after merging the terminal states. The value of node $\node_{31}$ is $\edgeval_{\node_{31}} = \vale(\nodelbvec) = \vale((-5,-4) = 0$, since the only feasible follower's solution for such a state is $\yvec = (0,0)$. \hfill $\square$ 
 
\end{example}

\subsection{Strengthening Terminal Node Values}
\label{subsec:robust}

As suggested by Example \ref{ex:approximation_example_construction}, the approximation procedure relaxes state information iteratively during construction to ensure the network remains small, and could lead to potentially loose upper bounds. In this section, we further leverage the network structure to strengthen node terminal values of an approximate value network $\net_A$. Specifically, let $\node \in \nodeset_{\nld+1}$ be a terminal node in $\net_A$, and consider the sub-network $\net(\node) = (\nodeset(\node), \edgeset(\node))$ induced by the $(\initnode, \node)$-paths in $\net_A$; formally, 
\begin{align*}
	\nodeset(\node) &:= \left\{ 
		v \in \nodeset \colon \exists\,\, \textnormal{path} \,\,\, (\initnode, \node_1, \x_1), \dots, (\node_{\nld}, \x_{\nld}, \node_{\nld+1}) \in \net_{A}, \node_{\nld+1} = \node, \,\, \node_k = v 
		\,\,\, \textnormal{for some $k$}
	\right \},\\
	\edgeset(\node) &:= \left\{
		(v, v', \x) \colon v, v' \in \nodeset(\node), \,\, (v, v', \x) \in \edgeset
	\right\}.
\end{align*}
For example, the shaded arcs in Figure \ref{fig:approx_dd}-(c) depict $\net(\node_{31})$. Analogous to the original system, the sub-network $\net(\node)$ is associated with a subset of leader's decisions represented by their path labels,
\begin{align*}
	\mathscr{X}(\node) = 
		\{ 
			\xvec \in \{0,1\}^{\nld} 
			\, \colon \,
			\exists
			\, 
			(\initnode, \node_1, \x_1),
			(\node_2, \node_3, \x_2),
			\dots,
			(\node_{\nld}, \node, \x_{\nld}) 
			\in
			\net(\node)
	  	\},
\end{align*}
which is distinct per terminal node $\node$. In view of the objective feasibility of Definition \ref{def:avdd}-(c2), the tighest terminal value of $\node$ that ensures objective feasibility is  therefore
\begin{align}
	\hat{\edgeval}_{\node} 
	&:= 
	\max_{\xvec \in \mathscr{X}(\node)} \val(\xvec) \nonumber \\
	&=
	\max_{\xvec \in \mathscr{X}(\node)} 
	\min_{\yvec} 
	\left\{ 
		g(\yvec) 
		\,\colon\,
		\Bf \yvec \ge \rhsf - \Al \xvec, \;\; \yvec \in \yvecset
	\right\}. \label{model:robust}
\end{align}
Problem \eqref{model:robust} is often referred to as a network interdiction problem (e.g., \citealt{borrero2019sequential, fischetti2019interdiction}). In particular, the outer problem selects a path in $\net(\Node)$ to maximize gains against an inner problem (or adversary) where actions are constrained by such path. We refer to the survey by \cite{smith2020survey} for recent models and solution approaches. 

Using $\hat{\edgeval}_{\node}$ as a terminal value can significantly strengthen the quality of the bound, as illustrated below.

\begin{example}
    \label{example:boundstrengthening}
    Consider the sub-network $\net(\node_{31})$ represented by the shaded arcs in Figure \ref{fig:approx_dd}-(c). The upper bound of $0$ observed at node $\node_{31}$ is due to the relaxed state $\nodelbvec = (-5,-4)$, which cannot be derived by any of the paths $\mathscr{X}(\node_{31})$ and is therefore excessively relaxed. Alternatively, when solving \eqref{model:robust} by inspection, the largest upper bound can be obtained by setting $\xvec = (1,0,1)$; in such a case, $\val(\xvec) = -100$ with an optimal follower's solution $\yvec = (\y_1, \y_2) = (1,0)$, and hence we can set $\edgeval_{31} = -100$. This also shows that the improvement can be arbitrarily large by considering any desired negative coefficient for $\y_1$. \hfill $\square$
\end{example}

To strengthen the value of the terminal $\node$ encoding the hyperrectangle $[\nodelbvec, \nodeubvec]$, we apply a cutting-plane methodology to \eqref{model:robust} proposed by \cite{lozano2017b} that iteratively improves an upper bound to $\hat{\edgeval}_{\node}$, and can be halted at any given computational limit. Specifically, let $\yvecset^S = \{\yvec_1, \yvec_2, \dots \yvec_K\} \subseteq \yvecset$ be a sample of $K$ follower's solutions within the set $\yvecset$. Then,
\begin{align}
	\hat{\edgeval}_{\node} 
	&=
	\max_{\xvec \in \mathscr{X}(\node)} 
	\min_{\yvec} 
	\left\{ 
		g(\yvec) 
		\,\colon\,
		\Bf \yvec \ge \rhsf - \Al \xvec, \;\; \yvec \in \yvecset
	\right\} \nonumber \\
	&\le
	\max_{\xvec \in \mathscr{X}(\node)} 
	\min_{\yvec} 
	\left\{ 
		g(\yvec) 
		\,\colon\,
		\Bf \yvec \ge \rhsf - \Al \xvec, \;\; \yvec \in \yvecset^S
	\right\}
	=: \tilde{\edgeval}_{\node},
	\label{eq:sampleinterdiction}
\end{align}
as the solution space of the follower's subproblem becomes more restricted when limiting the feasible space to samples. Since $\yvecset^S$ is discrete, we can reformulate \eqref{eq:sampleinterdiction} as a mixed-integer linear program (see Appendix \ref{app:formulations}). Our strengthening procedure can be written as follows, which is conducted for each node $\node$:
\begin{enumerate}
	\item Initialize $\yvecset^S$ by sampling $K > 0$ vectors from $\yvecset$ and solve \eqref{model:milpsample}, collecting its optimal solution $\xvec^*$ and value $\tilde{\edgeval}_{\node}$. 
	Set $\edgeval_{\node}$ to its new upper bound $\tilde{\edgeval}_{\node}$.
	
	\item Solve $\val(\xvec^*)$ and collect its optimal solution $\yvec^*$. If $\yvec^*$ is already in the sample set $\yvecset^S$, stop; otherwise, add to $\yvecset^S$ and repeat steps 1 and 2 until a desired computational limit.
\end{enumerate}

In particular, the procedure iteratively augments the sample set $\yvecset^S$ to better approximate $\yvecset$ by adding vectors $\yvec$ that are still feasible to $\xvec^*$, but were missed in the original sample set. If no such solution is found, then $\tilde{\edgeval}_{\node} = \hat{\edgeval}_{\node}$ and we obtain the tightest upper bound to the value of $\node$. Further, the procedure converges to the optimal value if $\yvecset$ is finite and discrete. 

However, for large networks, the robust problem \eqref{eq:sampleinterdiction} could still be challenging to optimize. We also observed significant bound improvements when considering a relaxed version that only accounts for the possible state set $[\nodelbvec, \nodeubvec]$ of node $\node$; i.e., note that 
\begin{align}
	\tilde{\edgeval}_{\node} 
	&=
		\max_{\xvec \in \mathscr{X}(\node)} 
		\min_{\yvec} 
		\left\{ 
			g(\yvec) 
			\,\colon\,
			\Bf \yvec \ge \rhsf - \Al \xvec, \;\; \yvec \in \yvecset^S
		\right\} \\
	&\le
		\max_{\xvec \colon \Al \xvec \in [\nodelbvec, \nodeubvec]} 
		\min_{\yvec} 
		\left\{ 
			g(\yvec) 
			\,\colon\,
			\Bf \yvec \ge \rhsf - \Al \xvec, \;\; \yvec \in \yvecset^S
		\right\}.
	\label{eq:sampleinterdictionhyperrect}
\end{align}
More specifically, we still obtain a valid bound if the set of states $\Al \xvec$ is restricted to be within the hyperrectangle $[\nodelbvec, \nodeubvec]$ of node $\node$, which leads to a much more tractable problem if $\nodeset$ is in the order of thousands, as we assess in \S\ref{sec:numericalstudy}. We also note that other improvements for efficiency are also possible, such as incorporate the leader-specific constraints \eqref{cons:linearleader} into the outer-problem of the formulation above.








\section{Numerical Study}
\label{sec:numericalstudy}

In this section, we conduct a computational study to evaluate the quality of the proposed relaxations and their impact on solution performance. In view of the wide availability of MIBLP solvers, our study focuses on insights for problems with linear constraints, linear objectives, and binary variables:
\begin{subequations}
	\label{model:miblp}
	\begin{align}
		\min_{\xvec, \yvec}  
			&\quad 
			\mathbf{c}^{\top} \xvec + \mathbf{p}^{\top} \yvec
				\label{miblp:obj} \\
		\textnormal{s.t.}
			&\quad
				\Gx \xvec + \Gy \yvec \ge \hvec, 
					\label{miblp:cons:1} \\
			&\quad
            \yvec \in \argmin_{\yvec'} 
            \left\{ 
                \mathbf{d}^{\top} \yvec' 
                \,\colon\,
                \Al \xvec + \Bf \yvec' \ge \rhsf, \;\; \yvec' \in \{0,1\}^{\nf}
            \right\}, \label{miblp:cons:2} \\ 
			&\quad
				\xvec \in \{0,1\}^{\nld}.
					\label{miblp:cons:3}
	\end{align}
\end{subequations}

The objective of this experimental analysis is to (a) assess the strength of the bounds provided by an approximate value network under varying input parameters; and (b) investigate how such bounds improve runtimes when incorporated into existing solvers, whenever applicable. The study considers the following state-of-the-art MIBLP solvers for the analysis:
\begin{enumerate}
	\item the branch-and-cut approach by \cite{tahernejad2020branch}, denoted by \MibS{} (v. 1.2.1), an open-source COIN-OR implementation that combines bilevel relaxations and enumerative techniques; and
	\item the branch-and-cut approach by \cite{fischetti2017new}, denoted by \BC{}, based on a large family of intersection cuts and associated branching strategies.
\end{enumerate}

We conduct our analysis on two instance classes, as follows. In \S\ref{sec:num:structured}, we investigate structured problems to understand trade-offs between network size, quality of the bound, and runtimes. In \S\ref{sec:num:benchmark}, we present results on a subset of a well-known bilevel benchmark proposed by \cite{fischetti2016intersection} and based on the MIPLIB 3.0 benchmark. 
Details concerning the numerical environment, solvers, and additional supporting tables are included in Appendix \ref{app:numericaldetails}. All source code and instances will be made available online.

\subsection{Structured Instances}
\label{sec:num:structured}

The strength of the bounds provided by an approximate value network $\net_A$ depends on the cardinality of set $\stateset' := \{\Al \xvec \colon \xvec \in \{0,1\}^{\nld} \}$, as its elements define states that are translated to nodes of $\net_A$. More precisely, if the budget $\budget$ is sufficiently large for the nodes of $\net_A$ to capture all vectors of $\stateset'$, the network representation and the resulting value-function reformulation \eqref{model:vfr_ext} would be exact. Otherwise, more states are potentially aggregated into the nodes of $\net_A$ during the merging process of Algorithm \ref{algo:approxcompilation}, leading to a looser relaxation with respect to terminal values due to larger intervals $[\nodelbvec,\nodeubvec]$.

To control for the cardinality of $\stateset'$, we generate random instances based on four parameters $\nld, \ncons, \alpha,$ and $\beta$ that limit the sparsity and variance of the elements of $\Al$, as follows:

\begin{itemize}
\item We consider $\nld \in \{25,50\}$ for the leader's vector dimension and $\ncons \in \{1,10,20\}$ for the number of interaction constraints. The follower's coefficients $\Gy$ and $\Bf$ are drawn independently from the discrete uniform distribution $\textnormal{U}(0,100)$. 

\item The leader's coefficients $\Al$ are set as zero with probability 0.8, and otherwise draw independently from the discrete uniform distribution $5 \times \textnormal{U}(0, \alpha)$, where $\alpha \in \{1, 3, 5\}$ controls the diversity of $\Al$. For example, with $\nld = 50$ and $\alpha = 1$, a row of $\Al$ has on average 10 elements within the set $\{0, 5\}$, i.e., each state component of $\stateset'$ is of the form $5k$ for $k=0,1,\dots,10$. This implies that there exist at most 11 possible values for every state component, and therefore $|\stateset'| \in \mathcal{O}(11^{\ncons})$ (see \S\ref{sec:sizestatebased}). For $\alpha=2$ and $\alpha=5$, we obtain $|\stateset'| \in \mathcal{O}(31^{\ncons})$ and $|\stateset'| \in \mathcal{O}(121^{\ncons})$, respectively (calculated, e.g., via a dynamic program that enumerates subset sums). The matrix $\Gx$ is generated analogously as $\Al$. 

\item The right-hand side of the interaction constraint is set as $\rhsf := \beta \times (\Al \xvec + \Bf \yvec )$, where $\beta \in \{0.1, 0.3, 0.5\}$ controls the constraint tightness. The vector $\hvec$ is generated analogously.

\item The leader's objective coefficients $\mathbf{c}$ and $\mathbf{p}$ are drawn independently at random from $\textnormal{U}(-100, -1)$, and the follower's objective coefficients $\mathbf{d}$ are drawn independently at random from $\textnormal{U}(-50,50)$. 
\end{itemize}

\smallskip
Five different problem instances were generated for each configuration $(\nld, \ncons, \alpha, \beta)$, for a total of $270$ instances. The reported gaps are with respect to the best-known upper bound obtained across all methods. We also remark that instances are expected to be more challenging for larger values of $\ncons$ and $\alpha$, as the number of possible states $\stateset'$ and follower's subproblem values increase significantly. 

Finally, we report the standard deviation in square brackets next to the corresponding average, when applicable; e.g., 20 [12.5] denotes an average of 20 with an observed standard deviation of 12.5.

\smallskip
\noindent \textit{Value network parameters.} The analysis considers two variants of an approximate value network: one using the standard methodology depicted by Algorithm \ref{algo:approxcompilation}, denoted by $\net_A$, and another applying the terminal value strengthening procedure from \S\ref{subsec:robust}, denoted by $\net^{+}_A$. The budget is set to $\budget = 50$ in all runs. For the node merging procedure of Algorithm \ref{algo:approxcompilation}, Step 7, we implemented a straightforward technique that performed well during fine-tuning, described as follows. First, if $\nodeset_{j+1} > |\budget|$, we order the nodes of the  $j+1$-th layer in decreasing order according to the longest path from the root node to every node in $\nodeset_{j+1}$, where the arc lengths are given by the leader's objective coefficient $\mathbf{c}$. Then, we merge pairs of nodes following this ordering until the width of the layer is less than or equal to $\budget$. Finally, variables in the compilation of $\net_A$ and $\net^+_A$ are ordered by the sum of their coefficients in $\Al$. 

For the terminal value strengthening of $\net^{+}_A$, we start with a small sample $\yvecset^{S}$ and run at most five iterations of the cutting-plane method described in \S\ref{subsec:robust}, considering the simplified version of the outer-optimization problem based on hyperrectangles. The details of how the initial samples are generated and supporting methods are in included in Appendix \ref{app:numericaldetails}.

\smallskip
\noindent \textit{Comparison to high-point relaxations and network sizes.} We first evaluate the improvements with respect to the high-point relaxation, referred to as \texttt{HPR}, and the impact of the reduction procedure to eliminate symmetric nodes (\S\ref{subsec:reduction}). In particular, \texttt{HPR} for \eqref{model:miblp} is defined by 
\begin{align}
	\min_{\xvec \in \{0,1\}^{\nld}, \yvec \in \{0,1\}^{\nfl}} 
	\left\{
		\mathbf{c}^{\top} \xvec + \mathbf{p}^{\top} \yvec
		\colon
		\Gx \xvec + \Gy \yvec \ge \hvec, \,\, \Al \xvec + \Bf \yvec \ge \rhsf
	\right\}.
	\label{model:miblp_hpr}
\end{align}
To assess the bound improvements with respect to \texttt{HPR}, we write the strengthened model for $\net$:
\begin{align}
	\min_{\xvec \in \{0,1\}^{\nld}, \yvec \in \{0,1\}^{\nfl}, \z} 
	\left\{
		\mathbf{c}^{\top} \xvec + \mathbf{p}^{\top} \yvec
		\colon
		\Gx \xvec + \Gy \yvec \ge \hvec, 
		\,\, \Al \xvec + \Bf \yvec \ge \rhsf,
		\,\,\z \le \mathbf{d}^{\top} \yvec,
		\,\, (\xvec, \z) \in \polynet(\net)
	\right\}.
	\label{model:hpr_net}
\end{align}
We refer to \DD{} and \DDMaxMin{} as the model \eqref{model:hpr_net} obtained when replacing $\net$ by 
$\net_A$ and $\net^+_A$, respectively. An optimal leader's solution $\xvec^*$ of \eqref{model:miblp_hpr} or \eqref{model:hpr_net} is integer and its value provides a lower bound $v^{\textnormal{LB}}$ to the optimal value $v^*$ of \eqref{model:miblp}. We denote 
the quantity $(v^* - v^{\textnormal{LB}})/v^*$ as the integer optimality gap of the model.

\begin{figure}[t!]
	\begin{subfigure}[b]{0.5\textwidth}
		\includegraphics[scale=0.45]{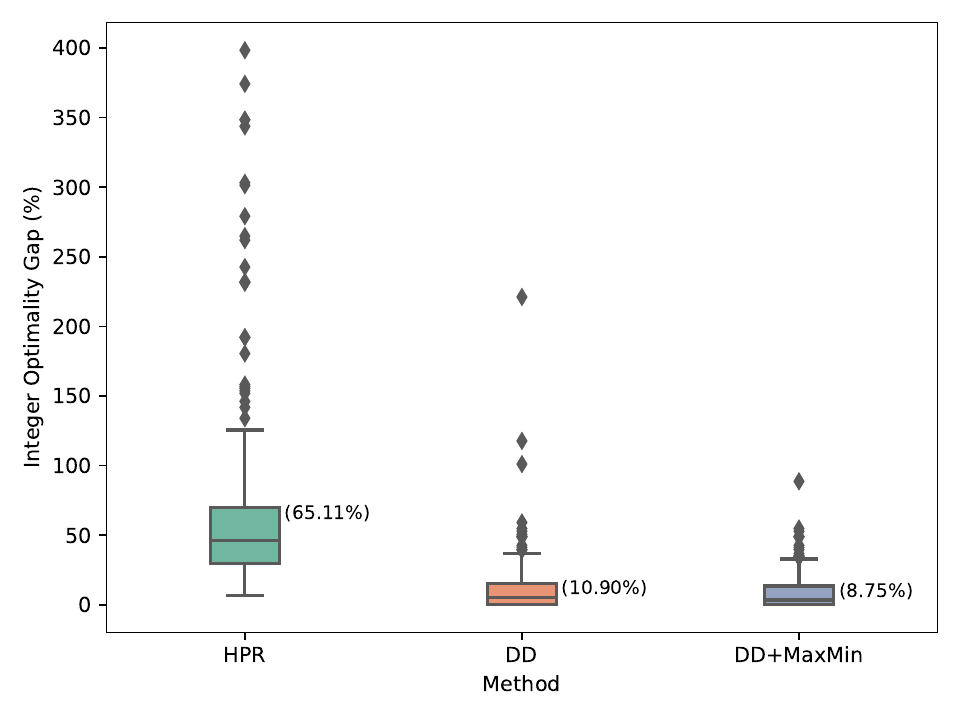} 
		\caption{Integer optimality gap.}
		\label{fig:integer_gap}
    \end{subfigure}
	\begin{subfigure}[b]{0.5\textwidth}
		\includegraphics[scale=0.45]{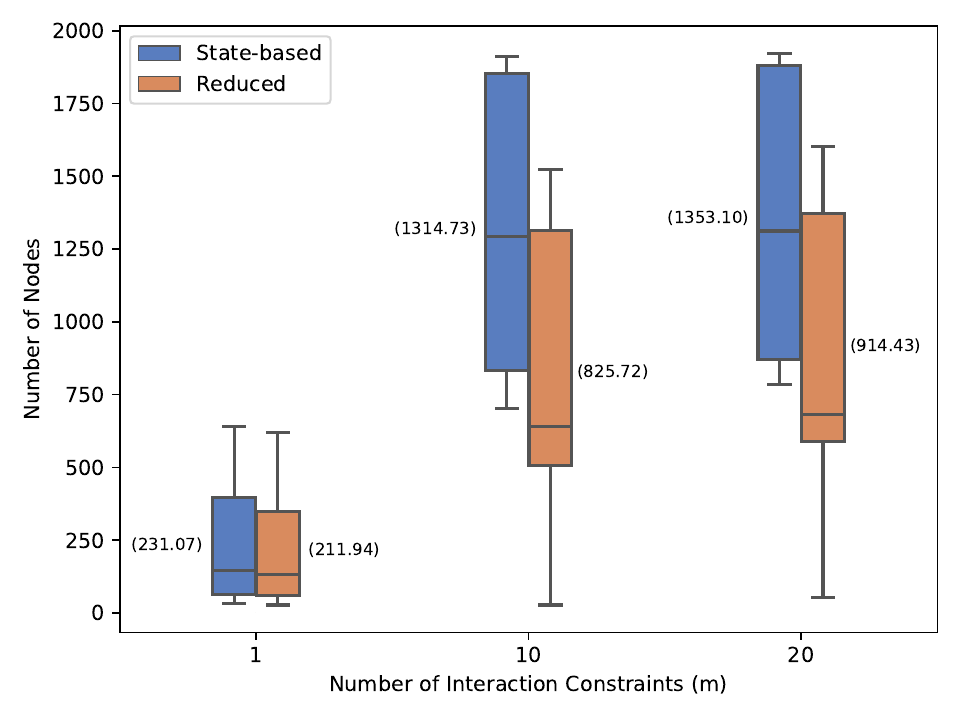} 
		\caption{Network sizes (\DD{} and \DDMaxMin{}).}
		\label{fig:reduction}
	\end{subfigure}	
    \caption{(Colored.) Box-and-whisker plots summarizing the distributions of integer optimality gaps and size of state-based and reduced networks. The box represent the interquantile range or IQR (difference between the 75\% and 25\% quantile), the lines within the boxes represent the median, the whiskers represent points within $1.5 \times$ IQR, and the remaining points are plotted individually. Values in parenthesis represent the averages.}
    \label{fig:integer_gap_reduction}
\end{figure}

Figure \ref{fig:integer_gap} depicts a box-and-whisker plot summarizing the distribution of the integer optimality gaps of each instance and model. Values in parenthesis next to each box depict averages. \texttt{HPR} provided an average optimality gap of 65.11\% [64.51\%], in contrast to an average of 10.89\% [19.84\%] for \DD{} and 8.75\% [12.49\%] for \DDMaxMin{}. Across all 270 test cases, \DD{} had a zero gap for all 90 instances with $\ncons=1$, an average gap of 12.97\% [11.78\%] for 90 instances with $\ncons=2$, and an average gap of 19.63\% [29.01\%] for the 90 instances with $\ncons=3$. \DDMaxMin{}, in turn, also had a zero optimality gap for $\ncons=1$, and presented generally stronger average optimality gaps of 11.86\% [11.68\%] and 14.36\% [14.63\%] for instances with $\ncons=2$ and $\ncons=3$, respectively. 

Figure \ref{fig:reduction} provides further intuition on the network sizes and impact of reduction for \DDMaxMin{} for each value of $\ncons$, summarizing the distribution of the number of nodes for the state-based and the reduced network. For $\ncons=1$, reduction had a small marginal effect on network sizes, as terminal node values were mostly distinct for each possible state (a scalar when $\ncons=1$). We also observed that all networks were exact in this case, i.e., no state aggregation occured for all 90 instances with $\ncons=1$. In constrast, all instances for $\ncons=10$ and $\ncons=20$ reached their budget $\budget$ and presented similar final network sizes, as also reflected in the box-plots. Nonetheless, reduction was more effective in these circumstances; on average, the network size reduced by 3.42 times [6.95] for $\ncons=2$, and $2.35$ times [4.88] for $\ncons = 3$. As expected, the variance parameter $\alpha$ also impacted reduction effectiveness: for $\alpha = 1$, $\alpha = 3$, and $\alpha=5$, the network size reduced by 3.72 times [7.81], 1.84 times [3.23], and 1.31 times [0.21], as the number of distinct follower's subproblem values is expected to increase with larger $\alpha$.

\smallskip
\noindent \textit{Comparison with exact approaches}. Next, we evaluate whether an approximate value network potentially captures distinct structure in comparison to \texttt{MibS} and \BC{}, which implement different state-of-the-art relaxation mechanisms. Our objecive is to assess whether the approximate network could be complementary to such relaxations. To this end, we compare total runtime performance with respect to an exact cutting-plane variant building on \DDMaxMin{}. More precisely, let $(\xvec^*, \yvec^*, \z^*)$ be the optimal solution of \DDMaxMin{}, i.e., formulation \eqref{model:hpr_net}. If $\z^* > \val(\xvec^*)$, the solution $(\xvec^*, \yvec^*, \z^*)$ is not feasible to \eqref{model:miblp} because $\yvec^*$ violates \eqref{miblp:cons:2}; that is, the terminal node value associated with $\xvec^*$ in $\net^+_A$ is overly restricted. However, we can improve the relaxation by adding an optimality cut of the form 
$$\x = \x^* \Rightarrow \z = \val(\x^*),$$
to \eqref{model:hpr_net}, which specifies the correct value of $\z$ when setting the leader's decision to $\x^*$. We can then resolve the model and repeat the procedure until a solution such that $\z^* = \val(\x^*)$ is found, which proves optimality. To linearize the above cut, we implement the blocking cuts from \cite{Lozano2017} (see Appendix \ref{app:numericaldetails} for the description of the cuts). We have also implemented \texttt{HPR} with the same cutting-plane procedure; because runtimes were several orders of magnitude larger than the other methods, we ommitted its analysis. 

Let \DDMaxMinSep{} denote this exact separation method. Figure \ref{fig:runtimes} compare the runtimes for \DDMaxMinSep{}, \texttt{MibS}, and \BC{} considering a time limit of 3,600 seconds (one hour), depicted in dashed lines in the scatter plots. The runtime for \DDMaxMinSep{} account for network construction, reduction, terminal value strengthening procedure, and the separation. In particular, network construction and reduction accounted, on average, for 24.25\% [23.89\%] of the total time, while the terminal value strengthening procedure accounted for 28.76\% [24.18\%] of the total time. The remaining portion of the time was largely spent solving the underlying mixed-integer linear programs.

Across the 270 instances, \texttt{MibS}, \BC{}, and \DDMaxMinSep{} solved to optimality 119, 190, and 239 instances, respectively. For the instances solved both by \texttt{MibS} and \DDMaxMinSep{} in Figure \ref{fig:runtime_mibs} (in total 119 cases), \texttt{MibS} outperformed \DDMaxMinSep{} in six instances of relatively small size ($\nld=25$, $\ncons=1$), due to the bottleneck in constructing the network. Instances were also much more difficult for larger $\ncons$. For \MibS{}, the average runtimes for $\ncons=1$ (77 instances), $\ncons=10$ (22 instances), and $\ncons=20$ (20 instances) were 130.93 seconds [298.50s], 1,048.50 seconds [970.55s], and 1,264.75 seconds [753.9s], respectively. For \DDMaxMinSep{}, the average runtimes were 0.88 seconds [0.35s], 14.05 seconds [13.95s], and 13.04 seconds [12.00s], respectively. We also note that \DDMaxMinSep{} was exact for $\ncons=1$ because of the small number of states that could be captured within the available budget $\budget = 50$, while the large number of interaction constraints when $\ncons=20$ was overall more challenging to all methodologies.

\begin{figure}[t!]
	\begin{subfigure}[b]{0.5\textwidth}
		\includegraphics[scale=0.45]{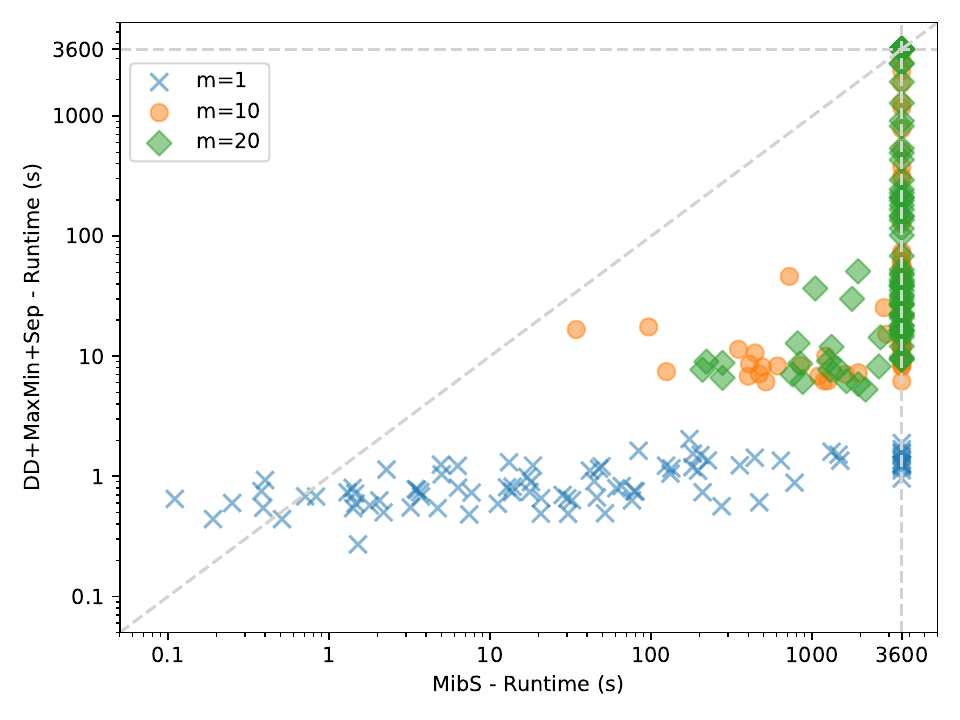} 
		\caption{Comparison with \MibS{}.}
		\label{fig:runtime_mibs}
    \end{subfigure}
	\begin{subfigure}[b]{0.5\textwidth}
		\includegraphics[scale=0.45]{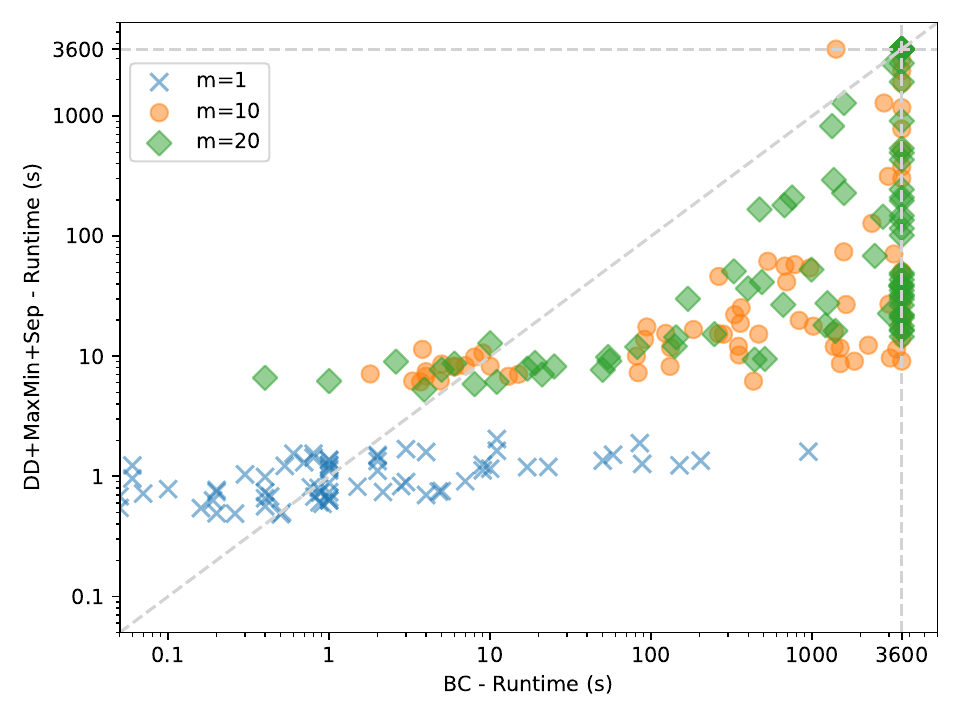} 
		\caption{Comparison with \BC{}.}
		\label{fig:runtime_bc}
	\end{subfigure}	
    \caption{(Colored.) Runtime comparions between \texttt{MibS}, \BC{}, and the exact \DDMaxMinSep{}. The $y$-axis is in logarithmic scale in both plots.}
    \label{fig:runtimes}
\end{figure}

In turn, for the instances solved both by \BC{} and \DDMaxMinSep{} in Figure \ref{fig:runtime_bc} (in total 189 cases), the average runtime for \BC{} for $\ncons=1$ (90 instances), $\ncons=10$ (56 instances), and $\ncons=20$ (43 instances) were 19.35 seconds [103.37s], 779.93 seconds [1,003.50s], and 731.64 seconds [975.74s], respectively. For \DDMaxMinSep{}, the average runtimes were 0.95 seconds [0.37s], 48.91 seconds [173.84s], and 155.27 seconds [462.12s], respectively. Although average times were significantly lower for \DDMaxMinSep{}, \BC{} outperformed \DDMaxMinSep{} in 70 instances. These were also generally smaller cases -- 52 instances with $\nld = 25$, and 50 instances with $\ncons=1$ -- that were solved by both methods in less than 40 seconds, but where the mixed-integer programs of \DDMaxMinSep{} presented a bottleneck in runtime. 

To provide further intuition on the parameters that result in a stronger \DDMaxMinSep{} performance, Figure \ref{fig:range_rhs_dd} depicts box-and-whisker runtime plots for different values of the variance parameter $\alpha$ and constraint tightness $\beta$. Only the 238 instances solved within the time limit of 3,600 seconds are considered. Specifically for the $\alpha$ comparison in Figure \ref{fig:range_dd}, the average runtimes for $\alpha=3$ and $\alpha=5$ increase by 6.32 times and 6.23 times, respectively, over the times for $\alpha=1$. For larger $\alpha$, the relaxation is weaker as the number of states increases and more nodes are merged; for these instances, the average integer optimality gap of \DDMaxMinSep{} for $\alpha=1$ (90 instances), $\alpha=3$ (79 instances), and $\alpha=5$ (70 instances) are 3.03\% [5.02\%], 8.08\% [10.67\%], and 10.53\% [15.02\%], respectively.

\begin{figure}[t!]
	\begin{subfigure}[b]{0.5\textwidth}
		\includegraphics[scale=0.45]{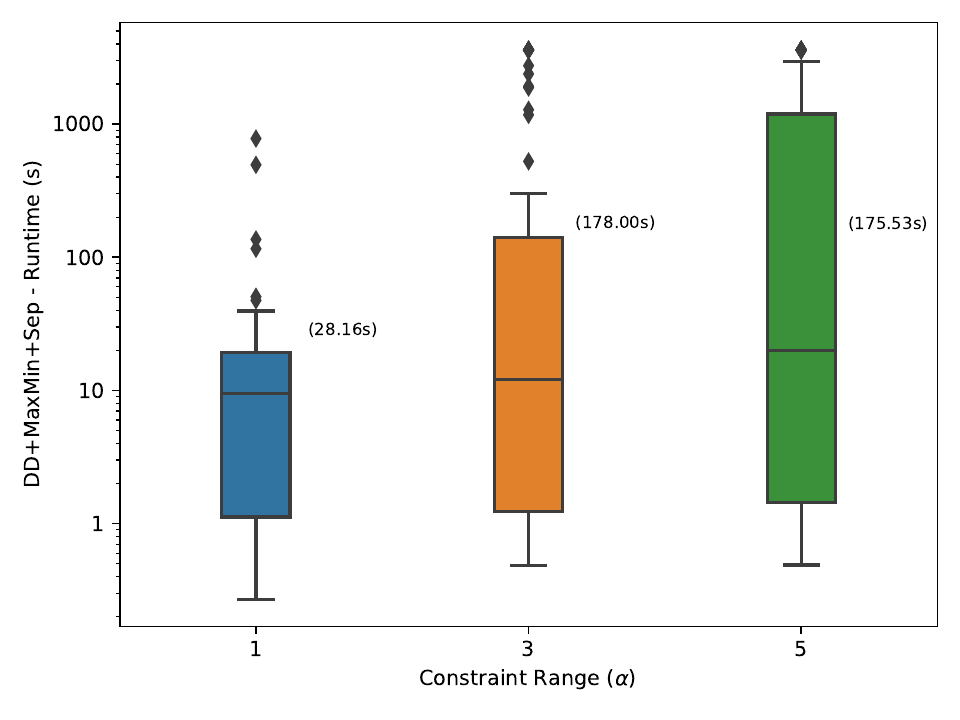} 
		\caption{\DDMaxMinSep{} runtimes for $\alpha=1,3,5$.}
		\label{fig:range_dd}
    \end{subfigure}
	\begin{subfigure}[b]{0.5\textwidth}
		\includegraphics[scale=0.45]{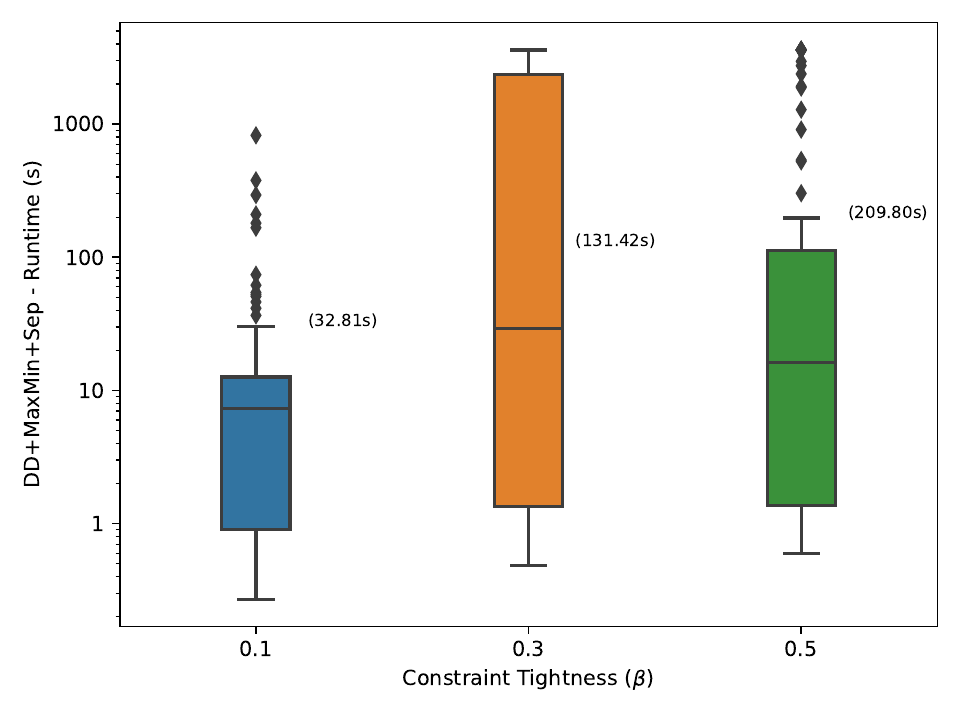} 
		\caption{\DDMaxMinSep{} runtimes for $\beta=.1,.3,.5$.}
		\label{fig:rhs_dd}
	\end{subfigure}	
    \caption{(Colored.) Runtime box-and-whisker plots for the exact variant of \DDMaxMinSep{}. The $y$-axis is in logarithmic scale in both plots. Times refer to instances solved up to the time limit of 3,600 seconds (238 cases).}
    \label{fig:range_rhs_dd}
\end{figure}

Analogously, Figure \ref{fig:rhs_dd} compares the distribution of runtimes for varying constraint tightness $\beta$. Runtimes were generally smaller for the extreme cases $\beta=0.1$ (90 instances) and $\beta=0.5$ (81 instances) when evaluating against $\beta=0.3$ (69 instances), as the number of feasible solutions for both the leader and follower are more restricted and relaxed if $\beta$ is smaller or larger, respectively; the results suggest that $\beta=0.3$ represents a phase transition when the problem becomes more difficult.

\subsection{Benchmark Instances and Integrating Relaxations}
\label{sec:num:benchmark}

We next assess performance when providing the stronger relaxations as input models to existing MIBLP solvers, verifying runtime gains from the complementary relaxations. For this analysis, we experiment on a well-known set of discrete bilevel problem instances proposed by \citet{fischetti2016intersection} and based on the MIPLIB 3.0 benchmark. We consider the subset of 33 instances (out of 57) having less than or equal to 3,000 variables and 3,000 constraints, as otherwise the value networks were excessively large to provide measurable runtime benefits. The experimental design and network parameters are the same as in \S\ref{sec:num:structured}, with the exception of the budget $\budget$: $\budget$ is set to $50$ if $\nld \in [0, 150]$, set to $25$ if $\nld \in (150, 300]$, set to $16$ if $\nld \in (300,500]$, set to $8$ if $\nld \in (500,1000]$, and set to $4$ if $\nld > 1000$. 

The instances and their sizes are described in Table \ref{tab:DiscreteLB}. For reference, the table also reports the integer optimality gaps obtained when solving \texttt{HPR} and \DDMaxMin{}, i.e., 
formulation \eqref{model:miblp_hpr} and formulation \eqref{model:hpr_net} with $\net^+_A$, respectively. The columns $\alpha$, $\tau$, and BKS denote number of different coefficient values of the interaction constraints, the percentage of non-zeros elements of $\Al$ (rounded down), and the best-known solution obtained by any method in 3,600 seconds, respectively (see runtime performance below). The other columns are defined as above. In particular, we observe that \DDMaxMin{} proved optimality for 18 instances out of the 33 cases, and improved the gap for three instances (\texttt{p0201-0.5}, \texttt{p0201-0.9} and \texttt{p0548-0.9}). Similar to the insights from \S\ref{sec:num:structured}, instances that were solved by \DDMaxMin{} or that resulted in better gaps generally had smaller $\nld$ (e.g., \texttt{stein27-0.9}), or smaller $\alpha$ and $\tau$ (e.g., \texttt{p0282-0.1}). 

\begin{table}[!t]
	\centering
	\scriptsize
	\caption{Integer Optimality Gaps between \texttt{HPR} and \DDMaxMin.}
	\label{tab:DiscreteLB}
	\begin{tabular}{rrrrrrrrrrrrr}
	\hline
	\multicolumn{ 1}{c}{Instance} & \multicolumn{ 1}{c}{$\nld$} & \multicolumn{ 1}{c}{$\nfl$} & \multicolumn{ 1}{c}{$m$} & \multicolumn{ 1}{c}{$\alpha$} & \multicolumn{ 1}{c}{$\tau$} & \multicolumn{ 1}{c}{BKS} &             \multicolumn{ 3}{c}{\texttt{HPR}} &              \multicolumn{ 3}{c}{\DDMaxMin} \\
	
	\multicolumn{ 1}{c}{} & \multicolumn{ 1}{c}{} & \multicolumn{ 1}{c}{} & \multicolumn{ 1}{c}{} & \multicolumn{ 1}{c}{} & \multicolumn{ 1}{c}{} & \multicolumn{ 1}{c}{} &   Time (s) &         LB &        Gap &   Time (s) &         LB &        Gap \\
	\hline

enigma-0.1 &         90 &         10 &         42 &        144 &     17.3\% &          0 &          0 &          0 &        0\% &          8 &          0 &        0\% \\

enigma-0.5 &         50 &         50 &         42 &         72 &     17.3\% &          0 &          0 &          0 &        0\% &          9 &          0 &        0\% \\

enigma-0.9 &         10 &         90 &         42 &         18 &     17.3\% &          0 &          0 &          0 &        0\% &         16 &          0 &        0\% \\
\hline
  lseu-0.1 &         81 &          8 &         28 &         30 &     28.5\% &       1120 &          0 &       1120 &        0\% &          8 &       1120 &        0\% \\

  lseu-0.5 &         45 &         44 &         28 &         17 &     23.6\% &       2263 &          0 &       1120 &       51\% &        891 &       1120 &       51\% \\

  lseu-0.9 &          9 &         80 &         28 &          4 &     39.7\% &       5838 &          0 &       1120 &       81\% &         31 &       5838 &        0\% \\
\hline
 p0033-0.1 &         30 &          3 &         16 &         15 &     26.7\% &       3089 &          0 &       3089 &        0\% &          6 &       3089 &        0\% \\

 p0033-0.5 &         17 &         16 &         16 &          4 &     23.0\% &       3095 &          0 &       3089 &        0\% &          8 &       3095 &        0\% \\

 p0033-0.9 &          4 &         29 &         16 &          4 &     50.0\% &       4679 &          0 &       3089 &       34\% &          2 &       4679 &        0\% \\
\hline
 p0201-0.1 &        181 &         20 &        133 &          4 &     17.8\% &      12225 &          0 &       7615 &       38\% &          8 &      12225 &        0\% \\

 p0201-0.5 &        101 &        100 &        133 &          4 &     12.8\% &      13635 &          0 &       7615 &       44\% &        529 &      11325 &       17\% \\

 p0201-0.9 &         21 &        180 &        133 &          4 &     16.3\% &      15025 &          0 &       7615 &       49\% &         24 &      14910 &        1\% \\
\hline
 p0282-0.1 &        254 &         28 &        241 &         21 &      4.0\% &     260781 &          0 &     258411 &        1\% &         14 &     258411 &        1\% \\

 p0282-0.5 &        141 &        141 &        241 &         27 &      4.3\% &     272659 &          0 &     258411 &        5\% &        165 &     258411 &        5\% \\

 p0282-0.9 &         29 &        253 &        241 &         16 &     44.5\% &     614837 &          0 &     258411 &       58\% &         20 &     614687 &        0\% \\
\hline
 p0548-0.1 &        494 &         54 &        176 &        108 &      3.9\% &      11067 &          0 &       8691 &       21\% &        112 &       8691 &       21\% \\

 p0548-0.5 &        274 &        274 &        176 &         77 &      1.5\% &      21750 &          0 &       8691 &       60\% &        112 &       8691 &       60\% \\

 p0548-0.9 &         55 &        493 &        176 &         35 &      2.7\% &      48852 &          0 &       8691 &       82\% &        516 &       9366 &       81\% \\
\hline
 p2756-0.1 &       2481 &        275 &        755 &         94 &      0.8\% &      11973 &          0 &       3124 &       74\% &        132 &       3124 &       74\% \\

 p2756-0.5 &       1378 &       1378 &        755 &         99 &      0.4\% &      23060 &          0 &       3124 &       86\% &        279 &       3124 &       86\% \\

 p2756-0.9 &        276 &       2480 &        755 &         63 &      0.6\% &      34471 &          0 &       3124 &       91\% &        818 &       3124 &       91\% \\
\hline
stein27-0.1 &         25 &          2 &        118 &          1 &     11.4\% &         18 &          0 &         18 &        0\% &          3 &         18 &        0\% \\

stein27-0.5 &         14 &         13 &        118 &          1 &     12.2\% &         19 &          0 &         18 &        5\% &          4 &         19 &        0\% \\

stein27-0.9 &          3 &         24 &        118 &          1 &     37.8\% &         24 &          0 &         18 &       25\% &          1 &         24 &        0\% \\

stein45-0.1 &         41 &          4 &        331 &          1 &      5.8\% &         30 &          3 &         30 &        0\% &         23 &         30 &        0\% \\

stein45-0.5 &         23 &         22 &        331 &          1 &      7.1\% &         32 &          3 &         30 &        6\% &         23 &         32 &        0\% \\

stein45-0.9 &          5 &         40 &        331 &          1 &     22.8\% &         40 &          3 &         30 &       25\% &          5 &         40 &        0\% \\
\hline
l152lav-0.1 &       1791 &        198 &        193 &         29 &      5.1\% &       4722 &          1 &       4722 &        0\% &        200 &       4722 &        0\% \\

l152lav-0.5 &        995 &        994 &        193 &         29 &      6.4\% &       4904 &          1 &       4722 &        4\% &       1140 &       4722 &        4\% \\

l152lav-0.9 &        199 &       1790 &        193 &         16 &      8.3\% &       5069 &          1 &       4722 &        7\% &       1205 &       4722 &        7\% \\
\hline
mod010-0.1 &       2390 &        265 &        291 &          2 &      3.1\% &       6554 &          0 &       6548 &        0\% &        348 &       6548 &        0\% \\

mod010-0.5 &       1328 &       1327 &        291 &          2 &      3.4\% &       6688 &          0 &       6548 &        2\% &       3248 &       6548 &        2\% \\

mod010-0.9 &        266 &       2389 &        291 &          2 &      4.8\% &       7441 &          0 &       6548 &       12\% &        623 &       6548 &       12\% \\

	\hline
			   &            &            &            &            &       &  {\bf Avg}         & {\bf 1}           &   & {\bf 26\%} &     {\bf 319}        &            & {\bf 16\%} \\
	\hline
	\end{tabular}  
\end{table}

We next report the performance of \BC{}, \MibS{}, and \DDMaxMinSep{} in this benchmark for a time limit of 3,600 seconds (one hour). To test the complementary of relaxations, we solve \eqref{model:miblp} augmented with the redundant flow constraints $\polynet(\net^{+}_A)$, providing the strengthened formulation as inputs to \MibS{} and \BC{}; we refer to these methods by \DDMibS{} and \DDBC{}, respectively. The results are presented in Table \ref{tab:MIPLIB}. In particular, the ``Time'' column refers to the runtime of convergence of the method, and ``\#Iter'' refers to the number of cutting-plane iterations until convergence of \DDMaxMinSep{}. The Gap column, in turn, presents the optimality gaps with respect to the best-known solutions reported in Table \ref{tab:DiscreteLB}, so as to provide an equal assessment of the strength of the final relaxation across methods. These best-known upper bounds are not input to solvers in advance. The runtimes for \DDMaxMinSep{}, \DDMibS{}, and \DDBC{} also account for network construction, reduction, and terminal value strengthening.

We highlight the following insights from Table \ref{tab:MIPLIB}.
\begin{itemize}

\item 
\texttt{enigma}, \texttt{stein}, and \texttt{p0033} instances, which are solved within a few seconds by both \MibS{} and \BC{}, are also solved by \DDMaxMinSep{} at its initial relaxation (i.e., no cuts were generated),
although at a higher runtime due to the network construction bottleneck and larger formulation sizes. The formulation size also impacted the higher runtimes for the integrations \DDMibS{} and \DDBC{}.

\item 
\BC{} generally presented the best performance for the three \texttt{lseu} instances, especially for \texttt{lseu-0.5} which was not solved by the other approaches. We also observed that \DDMaxMinSep{} and \DDBC{} presented relatively similar performance in \texttt{lseu-0.1} and slightly improved the runtime for \texttt{lseu-0.9}, which had a small number of leader's variables ($\nld=9$).
\end{itemize}
\begin{landscape}
	\begin{table}
	\centering
	\footnotesize
	\caption{Performance of the four exact algorithms on the MIPLIB instances by \cite{fischetti2016intersection}. Dash in the column ``Time(s)'' denotes 3,600 seconds (time limit).}
	\label{tab:MIPLIB}
	\begin{tabular}{|rrrrr|rr|rr|rrr|rr|rr|}
	\hline
	\multicolumn{ 1}{|c}{Instance} & \multicolumn{ 1}{c}{$\nld$} & \multicolumn{ 1}{c}{$\nfl$} & \multicolumn{ 1}{c}{$m$} & \multicolumn{ 1}{c}{\texttt{BKS}} & \multicolumn{ 2}{c}{\texttt{MibS}} & \multicolumn{ 2}{c}{\texttt{B\&C}} &             \multicolumn{ 3}{c}{\texttt{DD+MaxMin+Sep}} & \multicolumn{ 2}{c}{\texttt{MibS+DD}} & \multicolumn{ 2}{c|}{\texttt{B\&C+DD}} \\
	
	\multicolumn{ 1}{|c}{} & \multicolumn{ 1}{c}{} & \multicolumn{ 1}{c}{} & \multicolumn{ 1}{c}{} & \multicolumn{ 1}{c|}{} &   Time (s) &        Gap &   Time (s) &        Gap &   Time (s) &        Gap &       \# Iter &   Time (s) &        Gap &   Time (s) &        Gap \\
	\hline
enigma-0.1 &         90 &         10 &         42 &          0 &          1 &        0\% &          0 &        0\% &          6 &        0\% &          0 &         10 &        0\% &          6 &        0\% \\

enigma-0.5 &         50 &         50 &         42 &          0 &          1 &        0\% &          2 &        0\% &          5 &        0\% &          0 &          9 &        0\% &          7 &        0\% \\

enigma-0.9 &         10 &         90 &         42 &          0 &          1 &        0\% &          3 &        0\% &         10 &        0\% &          0 &         10 &        0\% &         13 &        0\% \\
\hline
  lseu-0.1 &         81 &          8 &         28 &       1120 &          2 &        0\% &          0 &        0\% &          7 &        0\% &          1 &          2 &        0\% &          0 &        0\% \\

  lseu-0.5 &         45 &         44 &         28 &       2263 &          - &       10\% &       2467 &        0\% &          - &       43\% &         30 &          - &       47\% &          - &        6\% \\

  lseu-0.9 &          9 &         80 &         28 &       5838 &          - &       18\% &         24 &        0\% &         17 &        0\% &          0 &         17 &        0\% &         19 &        0\% \\
\hline
 p0033-0.1 &         30 &          3 &         16 &       3089 &          0 &        0\% &          0 &        0\% &          4 &        0\% &          0 &          4 &        0\% &          4 &        0\% \\

 p0033-0.5 &         17 &         16 &         16 &       3095 &          0 &        0\% &          0 &        0\% &          6 &        0\% &          0 &         22 &        0\% &          7 &        0\% \\

 p0033-0.9 &          4 &         29 &         16 &       4679 &          0 &        0\% &          0 &        0\% &          2 &        0\% &          0 &          1 &        0\% &          1 &        0\% \\
\hline
 p0201-0.1 &        181 &         20 &        133 &      12225 &          - &       33\% &          - &       21\% &          7 &        0\% &          0 &          6 &        0\% &          6 &        0\% \\

 p0201-0.5 &        101 &        100 &        133 &      13635 &          - &       39\% &        539 &        0\% &        503 &        0\% &          0 &          - &       17\% &          - &       14\% \\

 p0201-0.9 &         21 &        180 &        133 &      15025 &          - &       43\% &          2 &        0\% &         19 &        0\% &          0 &          - &        1\% &         48 &        0\% \\
\hline
 p0282-0.1 &        254 &         28 &        241 &     260781 &          - &        1\% &        676 &        0\% &         43 &        0\% &          0 &          - &        1\% &         18 &        0\% \\

 p0282-0.5 &        141 &        141 &        241 &     272659 &          - &        5\% &       2938 &        0\% &        171 &        0\% &          0 &          - &        5\% &        233 &        0\% \\

 p0282-0.9 &         29 &        253 &        241 &     614837 &          - &       57\% &          - &       18\% &         15 &        0\% &          0 &          - &     0.03\% &         17 &        0\% \\
\hline
 p0548-0.1 &        494 &         54 &        176 &      11067 &          - &       33\% &          - &       17\% &          - &       19\% &         59 &          - &       36\% &          - &       17\% \\

 p0548-0.5 &        274 &        274 &        176 &      21750 &          - &       66\% &          - &       50\% &          - &       59\% &        135 &          - &       83\% &          - &       49\% \\

 p0548-0.9 &         55 &        493 &        176 &      48852 &          - &       85\% &          - &       62\% &          - &       78\% &         93 &          - &       81\% &          - &       66\% \\
\hline
 p2756-0.1 &       2481 &        275 &        755 &      11973 &          - &       77\% &          - &       72\% &          - &       74\% &        267 &          - &       77\% &          - &       71\% \\

 p2756-0.5 &       1378 &       1378 &        755 &      23060 &          - &       88\% &          - &       83\% &          - &       86\% &        222 &          - &       88\% &          - &       83\% \\

 p2756-0.9 &        276 &       2480 &        755 &      34471 &          - &       92\% &          - &       87\% &          - &       91\% &         42 &          - &       92\% &          - &       86\% \\
\hline
stein27-0.1 &         25 &          2 &        118 &         18 &          0 &        0\% &          0 &        0\% &          3 &        0\% &          0 &         19 &        0\% &          3 &        0\% \\

stein27-0.5 &         14 &         13 &        118 &         19 &          0 &        0\% &          0 &        0\% &          3 &        0\% &          0 &          3 &        0\% &          3 &        0\% \\

stein27-0.9 &          3 &         24 &        118 &         24 &          0 &        0\% &          0 &        0\% &          1 &        0\% &          0 &          1 &        0\% &          1 &        0\% \\

stein45-0.1 &         41 &          4 &        331 &         30 &         11 &        0\% &          1 &        0\% &         23 &        0\% &          0 &        550 &        0\% &         30 &        0\% \\

stein45-0.5 &         23 &         22 &        331 &         32 &          0 &        0\% &          0 &        0\% &          5 &        0\% &          0 &          5 &        0\% &          5 &        0\% \\

stein45-0.9 &          5 &         40 &        331 &         40 &          2 &        0\% &          0 &        0\% &          3 &        0\% &          0 &          3 &        0\% &          3 &        0\% \\
\hline
l152lav-0.1 &       1791 &        198 &        193 &       4722 &          3 &        0\% &          1 &        0\% &        112 &        0\% &          0 &        760 &        0\% &        140 &        0\% \\

l152lav-0.5 &        995 &        994 &        193 &       4904 &          - &        3\% &          - &        3\% &          - &        3\% &         28 &          - &        4\% &          - &        3\% \\

l152lav-0.9 &        199 &       1790 &        193 &       5069 &          - &        6\% &          - &        5\% &          - &        6\% &         41 &          - &        6\% &          - &        6\% \\
\hline
mod010-0.1 &       2390 &        265 &        291 &       6554 &          1 &        0\% &          2 &        0\% &       1069 &        0\% &         23 &          - &        0.1\% &        246 &        0\% \\

mod010-0.5 &       1328 &       1327 &        291 &       6688 &          - &        2\% &          - &        2\% &          - &        2\% &          8 &          - &        2\% &          - &        2\% \\

mod010-0.9 &        266 &       2389 &        291 &       7441 &          - &       11\% &          - &       12\% &          - &       12\% &         66 &          - &       12\% &          - &       12\% \\

	\hline
			   &            &            &     {\bf } & {\bf Avg. (\# solved)} & {\bf 1964} & {\bf 20.3\% (15)} & {\bf 1511} & {\bf 13.1\% (21)} & {\bf 1262} & {\bf 14.4\% (22)} &   {\bf 31} & {\bf 1898} & {\bf 16.8\% (17)} & {\bf 1334} & {\bf 12.6\% (21)} \\
	\hline
	\end{tabular}  
	
	\end{table}
\end{landscape}
\begin{itemize}
\item 
The network model presented measurable improvements for the difficult \texttt{p}-instance classes. In total, \DDMaxMinSep{},  \DDBC{}, and  \BC{} solved 6, 5, and 1 instance out 12, respectively, while \MibS{} presented a positive optimality gap for the 12 cases. We observe that the integration \DDMibS{} significantly reduced the final optimality gaps for \MibS{} on \texttt{p0201} and \texttt{p0282} instances. The integration \DDBC{}, in turn, presented at least one order of magnitude lower runtimes than \BC{} in \texttt{p0282}, and slightly better optimality gaps for all \texttt{p2756} instances.

\item 
All methods performed similarly on the large-scale \texttt{ll52-lab} and \texttt{mod010} instance classes, with no clear benefits or trade-offs observed in all approaches.
\end{itemize}

On average for these instance classes, \DDBC{} improved runtime and final gap by 11.7\% (1,334s vs 1,511s) and 3\% (12.6\% vs. 13.1\%) with respect to \BC{}, respectively. Similarly, \DDMibS{} improved the runtime of \MibS{} slightly by 3\% (1,898s vs. 1,964s) but its gap by 17\% (16.8\% vs. 20.3\%). Finally, although \DDMaxMinSep{} implements a straightforward separation procedure, the algorithm presented a surprisingly similar performance to \DDBC{} and \DDMibS{}, with an average runtime of 1,262 seconds, 14\% optimality gap, and 22 instances solved.



\section{Conclusion}
\label{sec:conclusion}

In this paper, we propose a network-based representation of the value function 
$\val$ for bilevel problems with linear interaction constraints and discrete leader's decisions. Our approach operates on a projection of the leader's decision into \textit{states} that map to distinct follower's subproblem evaluations, each captured by a node of the network. Evaluating the value function reduces to solving a flow linear program over the proposed network, which we show to define the convex hull of the graph of $\val$. For cases where the network is large, we provide a state-aggregation procedure that can be parameterized to provide networks of any desired maximum size. Our numerical experiments compare and incorporate the network encodings with state-of-the-art MIBLP solvers. Specifically, our improved high-point relaxations showed that strong performance correlated with the sparsity and range of coefficients in the interaction matrix $\Al$. By combining our relaxations with existing solvers, we also observed runtime and gap reductions for the branch-and-cut-based methods, solving additional instances within a one-hour limit.

The work also opens new avenues for research in both theory and methodology for discrete bilevel problems. For instance, it would be meaningful to investigate cuts that can be extracted from the networks, similar to \cite{davarnia2020outer}, for cases where the network is too large. Other areas include leveraging specific follower's subproblem structure to generate specialized approximate value networks with either stronger bounds or generally more compact.


\newpage


\bibliographystyle{ormsv080} 
\bibliography{references.bib} 



\clearpage
\begin{APPENDICES}

\section{Proofs} 
\label{app:proofs}
\begin{proof}{Proof of Proposition \ref{prop:convexhull}.}
Let $W := \{ \flowvec \in \mathbb{R}^{|\edgeset|} \colon \eqref{extform:cons:1}, \eqref{extform:cons:2}, \eqref{extform:cons:6}\}.$ The linear constraint set in $W$ corresponds to a network-flow model transferring one unit of flow from the initial node $\initnode$ to some terminal node $\termnode \in \nodeset_{\nld + 1}$; i.e., it is a path polytope. Thus, $W$ is bounded and integral \citep{AhujaEtal93}, and all extreme points are binaries. It follows that there is a one-to-one mapping between an extreme point $\flowvec^* \in W$ and a path $((\edge^*_1, \dots, \edge^*_{\nld}) =  (\node^*_1, \node^*_2, \x^*_1), (\node^*_2, \node^*_3, \x^*_2), \dots, (\node^*_{\nld}, \node^*_{\nld+1}, \x^*_{\nld})),
$
$\node^*_{j} \in \nodeset_{j}$ for all $j \in \augindexset$. Each path, in turn, has a one-to-one mapping with points of $\extval$ by construction. Moreover,
\begin{align*}
	\val(\xvec^*) = \edgeval_{\node^*_{\nld + 1}} = \sum_{\node \in \nodeset_{\nld+1}} \sum_{\edge \in \edgeset^{-}(\node)} \edgeval_{\node} \, \flow^*_{\edge},
\end{align*} 
where the first equality follows from Definition \ref{def:valuedd}-(C) and the second equality from the fact that edge $\edge^*_{j}$ belongs to the associated path if and only if $\flow_{\edge^*_{j}} = 1$. This implies that, for any extreme point $\flowvec^*$ in $W$, we can extend it to a valid point in $\polynet(\net)$ by setting $\z = \val(\xvec^*)$ and $\x_{j} = \x^*_{j}$ to satisfy constraints \eqref{extform:cons:3}-\eqref{extform:cons:5}. 

Finally, for appropriate matrices $\mathbf{D'}$ and $\mathbf{D'}$. we can rewrite the extension as
$$
	\polynet(\net) = \{ (\xvec, \z) \colon \exists \flowvec \in W \, \textnormal{s.t.} \, \xvec = \mathbf{D} \flowvec, \, \z = \mathbf{D'} \flowvec \}.
$$
It follows that any extreme point of $W$ has a one-to-one mapping to an extreme point of $\polynet(\net)$. More precisely, consider any convex combination $\flowvec' = \gamma \cdot \flowvec_1 + (1-\gamma) \cdot \flowvec_2$ for distinct $\flowvec_1,\flowvec_2 \in W$ and $\gamma \in (0,1)$. We must have $(\mathbf{D}\flowvec', \flowvec', \mathbf{D'}\flowvec') = \gamma \cdot (\mathbf{D}\flowvec_1, \flowvec_1, \mathbf{D'}\flowvec_1) + (1-\gamma) \cdot (\mathbf{D}\flowvec_2, \flowvec_2, \mathbf{D'}\flowvec_2)$. Thus, $(\xvec^*, \z)$ is an extreme point of $\polynet(\net)$ if and only if $\flowvec^*$ is an extreme point of $W$, since $\z = \val(\xvec^*)$. \hfill $\blacksquare$
\end{proof}
\smallskip
\begin{proof}{Proof of Lemma \ref{lem:statevalue}.}
	The conditions established by Definition \ref{def:valuedd}-(A) and (B) follow by construction. For the conditions in Definition \ref{def:valuedd}-(C), let $\xvec \in \{0,1\}^{\nld}$ and define the sequence $\state_1, \state_2, \dots, \state_{\nld + 1}$ such that $\state_1 = \zeros$ and $\state_{j+1} = \state_{j} + \al_{j} \x_{j}$ for $j = 1,\dots,\nld$. If $(\xvec, \val(\xvec)) \in \extval$, then we must have $\state_{\nld+1} \in \stateset_{\nld+1}$ and, by construction, $\state_{j} \in \stateset_{j}$ for $j \in \indexset$. That is, the sequence $\state_1, \dots, \state_{\nld+1}$ defines a path $(\state_{1}, \state_{2}, \x_{1}), (\state_{2}, \state_{3}, \x_{2}), \dots, (\state_{\nld}, \state_{\nld+1}, \x_{\nld})$ in $\net^S$ from the initial node $\initnode$ to some terminal node in $\nodeset^S_{\nld+1}$. From condition (c) of the lemma,
	$\edgeval^S_{\state_{\nld+1}} = \vale(\state_{\nld+1})$; that is, any pair $(\xvec, \val(\xvec)) \in \extval$ maps to a path in $\net^{S}$. 
	
	Conversely, consider any path $(\state_{1}, \state_{2}, \x_{1}), (\state_{2}, \state_{3}, \x_{2}), \dots, (\state_{\nld}, \state_{\nld+1}, \x_{\nld})$ in $\net^S$ from the initial node $\initnode$ to some terminal node in $\nodeset^S_{\nld+1}$. Then, by construction, $\state_{j} \in \nodeset^S_{j}$ for $j \in \augindexset$ and therefore $\state_{\nld+1} \in \stateset_{\nld+1}$. That is, $\state_{\nld+1} = \Al \xvec$ for the vector $\xvec = (\x_1, \dots, \x_{\nld})$. Since $\edgeval^S_{\state_{\nld+1}} = \vale(\state_{\nld+1}) =
	\vale(\Al \xvec) = \val(\xvec)$, we have that $(\xvec, \val(\xvec)) \in \extval$ holds, completing the proof.
	\hfill $\blacksquare$
\end{proof}
\smallskip
\begin{proof}{Proof of Proposition \ref{prop:minimality}.}
	The symmetry result follows by backward induction on the layer $j$. The basis case $j=\nld+1$ holds because nodes $u,v \in \nodeset_{\nld+1}$ are symmetric if and only if $\edgeval_{u} = \edgeval_{v}$, captured by (R1). For some $j < \nld+1$, suppose no two nodes in layers $\nodeset_{j+1}, \dots, \nodeset_{\nld+1}$ are symmetric after applying Algorithm \ref{algo:reduction}, but nodes $u,v \in \nodeset_{j}$ at layer $j$ are symmetric and violate condition (R2).
	
	There exists some $\x \in \{0,1\}$ such that $(u,u',\x)$, $(v,v',\x) \in \edgeset$ for nodes $u',v' \in \nodeset_{j+1}$, $u' \neq v'$. However, by Definition \ref{def:symmetryclass}, $u'$ and $v'$ are equivalent, otherwise there must exist some path starting at $u$ to a terminal node that either has distinct labels or terminal values than all paths starting at $v$ and also ending at a terminal node. This contradicts the induction hypothesis.
	
	Finally, uniqueness (up to isomorphism) follows since the resulting graph can be perceived as a multi-terminal binary decision diagram where no two nodes are equivalent, which is minimal and has no redundancies (Thm. 3.1.4, \citealt{wegener2000branching}).
	\hfill $\blacksquare$
\end{proof}
\smallskip
\begin{proof}{Proof of Lemma \ref{lem:statevalueapprox}.}
	Let $(\xvec, \val(\xvec)) \in \extval$ and, for any fixed $j \in \indexset$, $\state := \zeros + \sum_{j'=1}^{j-1}\al_{j'} \x_{j'}$ and $\state' = \state + \al_{j} \x_j$. By construction, $\state \in \stateset_j$ and $\state' \in \stateset_{j+1}$. Thus, by property (b), there exists 
	$(\Node, \Node', \x_j) \in \edgeset$ where $\Node \in \nodeset_j$, $\Node' \in \nodeset_{j+1}$, and $\state \in \Node$, $\state' \in \Node'$. Thus, since $j$ is arbitrary, there exists a path in $\net$ associated with labels $\xvec$, and Definition \ref{def:avdd}-(D1) is satisfied. Moreover, for the terminal $\Node \in \nodeset_{\nld+1}$ associated with this path,
	\begin{align*}
		\val(\xvec) = \vale(\Al \xvec) \le \max_{\state \in \Node} \vale(\state) = \edgeval_{\Node}
	\end{align*}
	because $\Al \xvec \in \Node$; i.e., Definition \ref{def:avdd}-(D2) is also satisfied.
	\hfill $\blacksquare$
\end{proof}

\smallskip

\section{Formulations for the Strengthening Procedure} 
\label{app:formulations}
Let $\mathscr{K} := \{1, \dots, K\}$ denote the index set of samples in $\yvecset^{S}$:
\begin{subequations}
\label{model:milpsample}
\begin{align}
	\tilde{\edgeval}_{\node} = \max_{\xvec} 
		\quad& 
		\delta
		\\
	\textnormal{s.t.}
		\quad&
			\delta \le g(\yvec_k)  + \vale(\nodelbvec)\sum_{j \in \{1,\dots,m\}}\gamma_{kj}, &\forall k \in \mathscr{K},
			\label{milpsample:cons:1} \\
		\quad&
			\al_j^{\top} \xvec + \mathbf{b}_j^{\top} \yvec_k  \le b_j - \epsilon\gamma_{kj} + \bar{a}_j(1 - \gamma_{kj})  , &\forall k \in \mathscr{K}, j \in \{1,\dots,m\}
			\label{milpsample:cons:2} \\
		\quad&
			\gamma_{kj} \in \{0,1\}, &\forall k \in \mathscr{K}, j \in \{1,\dots,m\} 
			\label{milpsample:cons:3} \\
		\quad&
			\xvec \in \mathscr{X}(\node).
			\label{milpsample:cons:4}
\end{align}
\end{subequations}
In formulation \eqref{model:milpsample}, $\epsilon > 0$ is a sufficiently small number that represents a feasibility tolerance, and 
$
    \bar{a}_j := \max_{\xvec \in \{0,1\}^{\nld}, \yvec \in \yvecset^{S}} \{\al_{j}^{\top} \xvec + \mathbf{b}^{\top}_j \yvec - b_j\}
$ 
is an upper bound on the component values of $\al_{j}^{\top} \xvec + \mathbf{b}_j^{\top} \yvec - b_j$ for $\xvec \in \{0,1\}^{\nld}$ and $\yvec \in \yvecset^{S}$. The binary variable $\gamma_{kj}$, defined by \eqref{milpsample:cons:3}, is one if the $k$-th sample $\yvec_k$ violates constraint $j \in \{1,\dots,m\}$ for the given $\xvec$, which is imposed by inequality 
\eqref{milpsample:cons:2}; that is, if $\gamma_{kj} = 1$, then $\al_{j}^{\top} \xvec + \mathbf{b}^{\top}_j \yvec  \le b_j - \epsilon < b_j$.  For all samples $\yvec_k$ that are feasible for $\xvec$, i.e., $\gamma_{kj} = 0$ for all $j \in \{1,\dots,m\}$, inequality \eqref{milpsample:cons:1} enforces the outer problem to be bounded by the follower's objective value, $g(\yvec_k)$. Thus, problem \eqref{model:milpsample} attempts to find a vector $\xvec$ that ``blocks'' feasible follower's solution to maximize $\delta$. 

To enforce a leader's decision $\xvec$ to be in an $(\initnode, \node)$-path as imposed by \eqref{milpsample:cons:4}, we can adapt the convex-hull formulation \eqref{model:extform} to consider only paths in $\net(\node)$. Specifically, denoting by $\initnode'$ the root node of $\net(\node)$, we have:
\begin{subequations}
	\label{model:subnet}
	\begin{align}
	\mathscr{X}(\node) := \bigg \{ \xvec 
	\colon \exists \flowvec \in \mathbb{R}^{|\edgeset(\node)|} \,\, \textnormal{s.t.} \,\,
	&\sum_{\edge \in \outedges(\initnode')} \flow_{\edge} = 1, 
			\label{subnet:cons:1} \\	
	&\sum_{\edge \in \outedges(v)} \flow_{\edge} - \sum_{\edge \in \inedges(v)} \flow_{\edge}
	= 0, &&\forall j \in \augindexset\setminus\{1\},  \forall v \in \nodeset_{j}(\node), 
		\label{subnet:cons:2} \\
	&\sum_{ \substack{ \edge = (v, v', 1) \colon v \in \nodeset_{j}(\node)}} \flow_{\edge} = \x_{j}, &&\forall j \in \indexset, 
		\label{subnet:cons:3} \\
	&\xvec \in \{0,1\}^{\nld}, \;\; \flowvec \ge 0 \,\,\, \}.
		\label{subnet:cons:6}
	\end{align}
\end{subequations}

\smallskip

\section{Additional Numerical Details} 
\label{app:numericaldetails}

In this section, we present detailed information about the computational setting, implementation, and additional results and supporting tables. 

\smallskip
\noindent \textit{Computational environment.} The experiments ran on an Intel(R) i7-10850H CPU at 2.70GHz with 32 GB of memory, each single thread. The value networks were implemented in Java. To ensure runtimes and bounds were comparable, all mathematical programs were modeled and solved using the same solver, ILOG IBM CPLEX 12.7, which is the version compatible with \BC{} by \cite{fischetti2017new}.

\smallskip
\noindent \textit{Solvers.} We ran our experiments on \MibS{} v.1.2.1 available at \url{https://github.com/coin-or/MibS}. A copy of \BC{} was obtained directly from authors of \cite{fischetti2017new}. For the value network implementation, all source code and instances will be made available online. 

\smallskip
\noindent \textit{Big-M Calculation.} Large constants $M$ are used in the cuts to provide an upper bound on the follower's value for any leader action. As a result, a valid value of $M$ can be directly computed by solving
\begin{align*}
	M := \max_{\xvec \in \{0,1\}^{\nld}, \yvec \in \{0,1\}^{\nfl}} 
	\left\{
	   \mathbf{d}^{\top} \yvec
		\colon
		\Gx \xvec + \Gy \yvec \ge \hvec, \,\, \Al \xvec + \Bf \yvec \ge \rhsf
	\right\},
	\label{model:miblp_hpr}
\end{align*}
which often produces weak bounds since the follower problem has a minimization sense, whereas the upper bound is obtained by maximizing the follower's objective. We instead use a similar procedure to the value strengthening described in Section \ref{subsec:robust} by noting that any follower value is bounded by 
\begin{align*}
	M
	&:= 
	\max_{\xvec \in \{0,1\}^{\nld}} \val(\xvec) \nonumber \\
	&=
	\max_{\xvec \in \{0,1\}^{\nld}} 
	\min_{\yvec} 
	\left\{ 
		g(\yvec) 
		\,\colon\,
		\Bf \yvec \ge \rhsf - \Al \xvec, \;\; \yvec \in \yvecset
	\right\}. 
\end{align*}
Since this is potentially a very challenging problem, we compute an upper bound as described in Section \ref{subsec:robust}, by using a cutting-plane algorithm and setting the stopping criterion to 50 iterations.    

\smallskip
\noindent \textit{Initial samples for $\yvecset^S$.} We obtain the initial set of samples $\yvecset^S$ from the cutting-plane algorithm described above by storing in $\yvecset^S$ all the distinct follower solutions encountered during the 50 iterations of the algorithm used to compute the $M$ values.    

\smallskip
\noindent \textit{Blocking cuts.} We implement the blocking cuts from \cite{Lozano2017}, where for each element of a sample of follower solutions $\hat\yvec \in \yvecset^S$ and each follower constraint we introduce an additional binary variable $w_{\hat\yvec,j}$ and the following constraints:
\begin{align*}
    &\mathbf{d}^{\top} \yvec \leq \mathbf{d}^{\top}\hat\yvec + \sum_{j \in \{1,\dots,m\}}Mw_{\hat\yvec,j}  &\forall \hat\yvec \in \yvecset^S \\
    &\mathbf{b}_j ^{\top} \hat\yvec + \al_{j} ^{\top} \xvec \leq b_j  -\epsilon w_{\hat\yvec,j} +(1-w_{\hat\yvec,j}) \bar{a}_j &\forall \hat\yvec \in \yvecset^S, j \in \{1,\dots,m\}, 
\end{align*}
where $\epsilon > 0$ is a sufficiently small number that represents a feasibility tolerance, and 
$
    \bar{a}_j := \max_{\xvec \in \{0,1\}^{\nld}, \yvec \in \yvecset^{S}} \{\al_{j}^{\top} \xvec + \mathbf{b}^{\top}_j \yvec - b_j\}
$ 
is an upper bound on the component values of $\al_{j}^{\top} \xvec + \mathbf{b}_j^{\top} \yvec - b_j$ for $\xvec \in \{0,1\}^{\nld}$ and $\yvec \in \yvecset^{S}$. For each element of the sample, if all the $w_{\hat\yvec,j}$ variables are equal to zero, then the inequality $\mathbf{d}^{\top} \yvec \leq \mathbf{d}^{\top}\hat\yvec$ holds. Otherwise, the first constraint is deactivated by the $M$ value and the second constraint ensures that $\al_{j}^{\top} \xvec + \mathbf{b}_j^{\top} \yvec < b_j$ for some $j \in \{1,\dots,m\}$. The initial sample set  $\yvecset^S$ is generated as described above and new cuts are generated iteratively from bilevel infeasible solutions $(\xvec^*, \yvec^*, \z^*)$, by adding an optimal follower response $\hat\yvec$ such that $\mathbf{d}^{\top}\hat\yvec = \val(\xvec^*)$ to $\yvecset^S$.   

\smallskip
\textit{Infeasible Follower's Subproblems.} A terminal node $\Node \in \nodeset_{\nld+1}$ with value $\edgeval_{\Node} = +\infty$ cannot necessarily be eliminated from an approximate value network, given that we are selecting the most restrictive constraint set among all potential states represented by $\Node$. However, further structure from $\xvecset(\cdot)$ or $\yvecset$ can be leveraged to rule out infeasible nodes while constructing the approximate network. In particular, let $\mathbf{V} \in \mathbb{R}^{\ncons}$ be the vector representing the minimum value that each state component must have to ensure the follower's subproblem is feasible, e.g., $V_i := \min_{\yvec \in \yvecset} \{ b_i - (\Bf \yvec)_i \}$ for all $i = 1, \dots, \ncons$, which can be calculated a priori. It follows that a node $(\nodelbvec, \nodeubvec)$ can be eliminated if $\nodelbvec < \mathbf{V}$. Analogous computations can be performed for $\nodeubvec$. We use these conditions in our numerical experiments in \S \ref{sec:numericalstudy}.

\smallskip
\noindent \textit{Additional Tables.} Table \ref{tab:extraoptgap} reports average integer optimality gaps for structured instances discussed in Section \ref{sec:num:structured} and Figure \ref{fig:integer_gap}-(a). Table \ref{tab:extraruntimes} reports average runtimes and number of instance solved for structured instances discussed in Figures \ref{fig:runtimes} and \ref{fig:range_rhs_dd}.   

\begin{table}
	\footnotesize
	\centering
	\caption{Average Integer Optimality Gaps for structured instances (in \%). Values in parenthesis are the standard deviation.}
	\label{tab:extraoptgap}
	\begin{tabular}{|cccr|r|r|r|}
		\hline
		\multicolumn{1}{|c}{$\nld$} & \multicolumn{1}{c}{$\ncons$} & \multicolumn{1}{c}{$\alpha$} & \multicolumn{1}{c}{$\beta$} & \multicolumn{1}{|r|}{\texttt{HPR}} & \multicolumn{1}{|r|}{\DD{}} & \multicolumn{1}{|r|}{\DDMaxMin{}} \\
		\hline
		25 & 1 & 1 & 0.10 & 58.09 (30.45) & 0.00 (0.00) & 0.00 (0.00) \\
		 &  &  & 0.30 & 63.58 (23.03) & 0.00 (0.00) & 0.00 (0.00) \\
		 &  &  & 0.50 & 51.43 (15.24) & 0.00 (0.00) & 0.00 (0.00) \\
		 &  & 3 & 0.10 & 61.54 (26.54) & 0.00 (0.00) & 0.00 (0.00) \\
		 &  &  & 0.30 & 62.40 (27.16) & 0.00 (0.00) & 0.00 (0.00) \\
		 &  &  & 0.50 & 51.99 (13.85) & 0.00 (0.00) & 0.00 (0.00) \\
		 &  & 5 & 0.10 & 67.15 (52.76) & 0.00 (0.00) & 0.00 (0.00) \\
		 &  &  & 0.30 & 68.53 (31.91) & 0.00 (0.00) & 0.00 (0.00) \\
		 &  &  & 0.50 & 48.06 (22.12) & 0.00 (0.00) & 0.00 (0.00) \\
		\hline
		 & 10 & 1 & 0.10 & 195.67 (148.62) & 3.99 (5.90) & 1.37 (3.06) \\
		 &  &  & 0.30 & 143.84 (97.11) & 10.37 (4.65) & 10.37 (4.65) \\
		 &  &  & 0.50 & 93.19 (55.35) & 1.97 (4.23) & 1.97 (4.23) \\
		 &  & 3 & 0.10 & 154.09 (131.03) & 7.12 (15.05) & 0.32 (0.72) \\
		 &  &  & 0.30 & 137.01 (89.99) & 26.56 (13.23) & 26.56 (13.23) \\
		 &  &  & 0.50 & 88.90 (46.36) & 6.95 (5.20) & 6.95 (5.20) \\
		 &  & 5 & 0.10 & 210.12 (239.60) & 23.39 (8.81) & 16.70 (12.09) \\
		 &  &  & 0.30 & 126.73 (77.69) & 32.68 (11.16) & 32.68 (11.16) \\
		 &  &  & 0.50 & 83.43 (45.31) & 9.58 (5.35) & 9.58 (5.35) \\
		\hline
		 & 20 & 1 & 0.10 & 309.32 (350.06) & 1.12 (2.51) & 1.12 (2.51) \\
		 &  &  & 0.30 & 74.34 (27.03) & 8.69 (8.62) & 7.70 (8.08) \\
		 &  &  & 0.50 & 61.55 (20.69) & 5.07 (3.92) & 5.07 (3.92) \\
		 &  & 3 & 0.10 & 182.48 (117.74) & 23.67 (21.08) & 11.61 (10.11) \\
		 &  &  & 0.30 & 67.47 (21.94) & 20.04 (8.82) & 20.04 (8.82) \\
		 &  &  & 0.50 & 56.41 (16.37) & 6.67 (3.15) & 6.67 (3.15) \\
		 &  & 5 & 0.10 & 214.39 (126.11) & 101.34 (79.35) & 28.79 (35.25) \\
		 &  &  & 0.30 & 69.02 (27.12) & 28.22 (15.52) & 28.22 (15.52) \\
		 &  &  & 0.50 & 52.34 (14.71) & 9.00 (4.27) & 9.00 (4.27) \\
		\hline
		50 & 1 & 1 & 0.10 & 33.47 (13.56) & 0.00 (0.00) & 0.00 (0.00) \\
		 &  &  & 0.30 & 34.94 (9.70) & 0.00 (0.00) & 0.00 (0.00) \\
		 &  &  & 0.50 & 30.93 (10.96) & 0.00 (0.00) & 0.00 (0.00) \\
		 &  & 3 & 0.10 & 37.27 (16.05) & 0.00 (0.00) & 0.00 (0.00) \\
		 &  &  & 0.30 & 34.89 (9.71) & 0.00 (0.00) & 0.00 (0.00) \\
		 &  &  & 0.50 & 31.49 (11.33) & 0.00 (0.00) & 0.00 (0.00) \\
		 &  & 5 & 0.10 & 36.61 (12.28) & 0.00 (0.00) & 0.00 (0.00) \\
		 &  &  & 0.30 & 36.81 (11.73) & 0.00 (0.00) & 0.00 (0.00) \\
		 &  &  & 0.50 & 31.69 (10.63) & 0.00 (0.00) & 0.00 (0.00) \\
		\hline
		 & 10 & 1 & 0.10 & 27.74 (11.36) & 5.30 (6.53) & 1.47 (1.41) \\
		 &  &  & 0.30 & 31.19 (8.77) & 5.22 (2.75) & 5.22 (2.75) \\
		 &  &  & 0.50 & 29.92 (8.08) & 3.05 (1.35) & 3.05 (1.35) \\
		 &  & 3 & 0.10 & 26.36 (12.37) & 21.09 (13.24) & 21.09 (13.24) \\
		 &  &  & 0.30 & 33.20 (9.24) & 15.07 (5.09) & 15.07 (5.09) \\
		 &  &  & 0.50 & 29.49 (8.09) & 7.65 (1.97) & 7.65 (1.97) \\
		 &  & 5 & 0.10 & 23.15 (15.71) & 19.87 (15.19) & 19.87 (15.19) \\
		 &  &  & 0.30 & 35.50 (13.50) & 21.70 (8.24) & 21.70 (8.24) \\
		 &  &  & 0.50 & 31.25 (10.46) & 11.92 (5.82) & 11.92 (5.82) \\
		\hline
		 & 20 & 1 & 0.10 & 38.72 (18.56) & 16.93 (9.04) & 7.55 (10.96) \\
		 &  &  & 0.30 & 38.13 (15.19) & 6.63 (7.83) & 6.63 (7.83) \\
		 &  &  & 0.50 & 40.62 (10.89) & 3.10 (1.77) & 3.10 (1.77) \\
		 &  & 3 & 0.10 & 30.47 (11.56) & 20.61 (6.87) & 20.61 (6.87) \\
		 &  &  & 0.30 & 45.43 (26.67) & 19.69 (14.52) & 19.69 (14.52) \\
		 &  &  & 0.50 & 41.30 (10.68) & 9.21 (2.32) & 9.21 (2.32) \\
		 &  & 5 & 0.10 & 33.94 (12.86) & 26.03 (9.50) & 26.03 (9.50) \\
		 &  &  & 0.30 & 45.45 (22.22) & 25.85 (15.47) & 25.85 (15.47) \\
		 &  &  & 0.50 & 50.10 (22.77) & 21.65 (18.87) & 21.65 (18.87) \\
		\hline
	\end{tabular}
\end{table}

\begin{table}
	\caption{Number of instances solved within the one-hour time limit (referred to as \#) and average runtime for each exact method. Values in parenthesis are the standard deviation. The averages only consider instances that are solved. }
	\label{tab:extraruntimes}
	\centering
	\footnotesize
	\begin{tabular}{|cccc|c r |r r |r r|}
		\hline
		& & & & \multicolumn{2}{c|}{\MibS{}} & \multicolumn{2}{c|}{\BC{}} & \multicolumn{2}{c|}{\DDMaxMinSep{}} \\
		\hline
		$\nld$ & $\ncons$ & $\alpha$ & $\beta$ & \# & Runtime (s) & \multicolumn{1}{c}{\#} & Runtime (s) & \# & Runtime (s) \\
		\hline
		25 & 1 & 1 & 0.10 & 5 & 1.50 (1.31) & 5 & 0.05 (0.08) & 5 & 0.50 (0.19) \\
		 &  &  & 0.30 & 5 & 8.82 (12.30) & 5 & 0.11 (0.17) & 5 & 0.63 (0.14) \\
		 &  &  & 0.50 & 5 & 5.50 (8.62) & 5 & 0.09 (0.19) & 5 & 0.65 (0.06) \\
		 &  & 3 & 0.10 & 5 & 4.08 (5.12) & 5 & 0.21 (0.32) & 5 & 0.65 (0.11) \\
		 &  &  & 0.30 & 5 & 107.17 (203.05) & 5 & 0.76 (0.24) & 5 & 0.57 (0.08) \\
		 &  &  & 0.50 & 5 & 24.93 (28.20) & 5 & 0.45 (0.95) & 5 & 0.87 (0.21) \\
		 &  & 5 & 0.10 & 5 & 6.52 (8.36) & 5 & 0.60 (0.34) & 5 & 0.67 (0.12) \\
		 &  &  & 0.30 & 5 & 105.13 (96.31) & 5 & 1.88 (1.91) & 5 & 0.68 (0.10) \\
		 &  &  & 0.50 & 5 & 5.77 (4.99) & 5 & 0.44 (0.63) & 5 & 0.97 (0.25) \\
		\hline
		 & 10 & 1 & 0.10 & 5 & 699.97 (405.55) & 5 & 5.18 (4.46) & 5 & 7.67 (2.11) \\
		 &  &  & 0.30 & 0 &          - & 4 & 669.50 (634.54) & 5 & 15.27 (4.41) \\
		 &  &  & 0.50 & 0 &          - & 4 & 1,326.50 (1,583.01) & 5 & 10.99 (2.69) \\
		 &  & 3 & 0.10 & 4 & 582.89 (475.63) & 5 & 6.32 (2.96) & 5 & 7.73 (1.85) \\
		 &  &  & 0.30 & 0 &          - & 5 & 1,568.60 (1,094.97) & 5 & 47.38 (48.65) \\
		 &  &  & 0.50 & 0 &          - & 5 & 1,397.20 (1,090.95) & 5 & 14.76 (7.15) \\
		 &  & 5 & 0.10 & 5 & 793.65 (486.72) & 5 & 7.80 (4.09) & 5 & 8.09 (0.59) \\
		 &  &  & 0.30 & 0 &          - & 4 & 1,520.75 (1,007.80) & 3 & 124.49 (164.66) \\
		 &  &  & 0.50 & 0 &          - & 5 & 1,565.80 (1,175.41) & 5 & 35.56 (26.39) \\
		\hline
		 & 20 & 1 & 0.10 & 5 & 1,015.74 (851.39) & 5 & 4.06 (3.84) & 5 & 7.72 (2.98) \\
		 &  &  & 0.30 & 0 &          - & 4 & 2,298.00 (1,172.91) & 5 & 27.02 (11.23) \\
		 &  &  & 0.50 & 0 &          - & 1 & 54.00 (-) & 5 & 16.70 (4.84) \\
		 &  & 3 & 0.10 & 5 & 919.04 (718.64) & 5 & 18.12 (18.78) & 5 & 7.49 (1.47) \\
		 &  &  & 0.30 & 0 &          - & 3 & 2,259.00 (610.34) & 5 & 349.25 (524.69) \\
		 &  &  & 0.50 & 0 &          - & 1 & 140.00 (-) & 5 & 42.12 (36.24) \\
		 &  & 5 & 0.10 & 5 & 1,380.23 (736.77) & 5 & 24.80 (18.31) & 5 & 8.18 (0.79) \\
		 &  &  & 0.30 & 0 &          - & 2 & 2,392.50 (1,154.71) & 4 & 906.33 (1,213.53) \\
		 &  &  & 0.50 & 0 &          - & 2 & 951.50 (410.83) & 5 & 155.44 (219.29) \\
		\hline
		50 & 1 & 1 & 0.10 & 5 & 24.90 (17.83) & 5 & 0.94 (1.20) & 5 & 0.76 (0.21) \\
		 &  &  & 0.30 & 3 & 716.73 (740.28) & 5 & 50.82 (87.74) & 5 & 1.25 (0.17) \\
		 &  &  & 0.50 & 5 & 106.85 (87.33) & 5 & 0.61 (0.54) & 5 & 1.17 (0.11) \\
		 &  & 3 & 0.10 & 5 & 79.14 (75.63) & 5 & 4.68 (4.05) & 5 & 0.91 (0.17) \\
		 &  &  & 0.30 & 1 & 181.68 (-) & 5 & 25.92 (35.62) & 5 & 1.27 (0.16) \\
		 &  &  & 0.50 & 3 & 227.10 (117.59) & 5 & 31.16 (66.44) & 5 & 1.39 (0.15) \\
		 &  & 5 & 0.10 & 5 & 252.97 (301.08) & 5 & 7.00 (6.82) & 5 & 1.10 (0.35) \\
		 &  &  & 0.30 & 2 & 1,388.13 (98.68) & 5 & 218.80 (409.25) & 5 & 1.52 (0.25) \\
		 &  &  & 0.50 & 3 & 208.68 (215.71) & 5 & 3.90 (4.22) & 5 & 1.61 (0.28) \\
		\hline
		 & 10 & 1 & 0.10 & 2 & 1,573.96 (506.81) & 5 & 171.80 (147.55) & 5 & 8.70 (2.22) \\
		 &  &  & 0.30 & 0 &          - & 0 &          - & 5 & 181.81 (333.46) \\
		 &  &  & 0.50 & 0 &          - & 0 &          - & 5 & 22.40 (9.25) \\
		 &  & 3 & 0.10 & 3 & 2,140.22 (1,229.99) & 5 & 566.80 (562.57) & 5 & 35.90 (24.41) \\
		 &  &  & 0.30 & 0 &          - & 0 &          - & 1 & 1,173.00 (-) \\
		 &  &  & 0.50 & 0 &          - & 0 &          - & 5 & 1,026.87 (1,035.24) \\
		 &  & 5 & 0.10 & 3 & 1,233.25 (2,023.04) & 4 & 443.75 (396.90) & 5 & 105.62 (153.60) \\
		 &  &  & 0.30 & 0 &          - & 0 &          - & 0 &          - \\
		 &  &  & 0.50 & 0 &          - & 1 & 2,793.00 (-) & 2 & 2,119.96 (1,184.01) \\
		\hline
		 & 20 & 1 & 0.10 & 2 & 1,989.04 (947.66) & 5 & 283.80 (184.56) & 5 & 12.08 (2.71) \\
		 &  &  & 0.30 & 0 &          - & 0 &          - & 5 & 126.33 (205.87) \\
		 &  &  & 0.50 & 0 &          - & 0 &          - & 5 & 65.18 (56.05) \\
		 &  & 3 & 0.10 & 2 & 1,848.39 (111.91) & 5 & 488.60 (308.34) & 5 & 68.28 (55.67) \\
		 &  &  & 0.30 & 0 &          - & 0 &          - & 0 &          - \\
		 &  &  & 0.50 & 0 &          - & 0 &          - & 3 & 1,620.50 (1,299.51) \\
		 &  & 5 & 0.10 & 1 & 1,045.12 (-) & 5 & 902.60 (425.40) & 5 & 308.54 (302.09) \\
		 &  &  & 0.30 & 0 &          - & 0 &          - & 0 &          - \\
		 &  &  & 0.50 & 0 &          - & 0 &          - & 1 & 908.67 (-) \\
		\hline
		\end{tabular}		
\end{table}

\end{APPENDICES}

\end{document}